\documentclass[11pt,reqno]{amsart}
 \usepackage{graphicx}
\usepackage{amscd,amsmath,amsopn,amssymb,amsthm,multicol}
\usepackage[color,matrix, all, 2cell]{xy}
\usepackage{amscd}
\usepackage{lscape}
\usepackage{slashed}
\usepackage{graphicx}
\usepackage{setspace}
\usepackage{upgreek}
\usepackage{textgreek}
\usepackage{enumerate}
\usepackage{color}
\usepackage{lscape}
\usepackage{tikz}
\usepackage{multirow}
\usepackage{cancel}
\usepackage{soul}
\usepackage{harmony}
\usepackage{comment}
\usepackage{wasysym}
\usepackage{mathrsfs}
\usepackage{mathtools}

\numberwithin{equation}{section}
\DeclareMathAlphabet{\mathscrbf}{OMS}{mdugm}{b}{n}

\DeclareMathOperator{\Ad}{Ad}

\DeclareMathOperator{\Ric}{\mathsf{Ric}}
\DeclareMathOperator{\ric}{\mathsf{ric}}

\DeclareMathOperator{\vol}{vol}
\DeclareMathOperator{\con}{\mathsf{c}}

\DeclareMathOperator{\dd}{d}

\newcommand{\fr}{\mathfrak}
\newcommand{\al}{\alpha}
\newcommand{\be}{\beta}
\newcommand{\eps}{\epsilon}
\newcommand{\wi}{\widetilde}
\newcommand{\bb}{\mathbb}
\newcommand{\mc}{\mathcal}

\DeclareMathOperator{\Ss}{\bb{S}}

\theoremstyle{plain}
\newtheorem{lemma}{Lemma} [section]
\newtheorem{theorem}[lemma]{Theorem}
\newtheorem{corol}[lemma] {Corollary}
\newtheorem{prop} [lemma]{Proposition}

\theoremstyle{definition}
\newtheorem{definition}[lemma] {Definition}
\newtheorem{example}[lemma] {Example}

\newtheorem{remark}[lemma] {Remark}
\newtheorem*{remark*}{Remark}

\definecolor{dark}{rgb}{0.18,0.18,0.68}
\definecolor{mydark}{rgb}{0.78,0.08,0.08}
\definecolor{crew}{rgb}{0.2,0.5,0.2}
\definecolor{mmg}{rgb}{0.31,0.50,0.23}
\definecolor{dblue}{rgb}{0.01,0.01,0.44}
\definecolor{red}{rgb}{0.57,0.11,0.15}
\definecolor{cobalt}{RGB}{61,99,181}
\usepackage[colorlinks,citecolor=cobalt,linkcolor=cobalt,urlcolor=cobalt,pdfpagemode=UseNone,backref = page]{hyperref}

\begin{document}

\title[Decomposable $(6, 5)$-solutions in eleven-dimensional supergravity]{Decomposable $(6, 5)$-solutions in eleven-dimensional supergravity}

  \author{Ioannis Chrysikos}
  \address{Faculty of Science, University of Hradec Kr\'alov\'e, Rokitanskeho 62, Hradec Kr\'alov\'e
50003, Czech Republic}
\email{ioannis.chrysikos(at)uhk.cz}

 \author{Anton Galaev}
\address{Faculty of Science, University of Hradec Kr\'alov\'e, Rokitanskeho 62, Hradec Kr\'alov\'e
50003, Czech Republic}
\email{anton.galaev(at)uhk.cz}
 
 

    \maketitle  

\begin{abstract}
This paper presents  a series of   constructions providing    eleven-dimensional  bosonic supergravity backgrounds. In particular, we treat Lorentzian manifolds   given in terms of twisted products of    six-dimensional     Lorentzian manifolds and  five-dimensional    Riemannian manifolds.  By considering a representative  flux 4-form  adapted to this setting, we  analyse the system of bosonic supergravity equations and   describe  the corresponding geometric constraints.  The new supergravity backgrounds appear  for special cases associated to the adapted flux 4-form. For example, we provide a relation of  eleven-dimensional supergravity   with Ricci-isotropic Walker manifolds, and illustrate  several results   in their terms  and in terms of Ricci-flat Riemannian manifolds. 
\end{abstract}


\section*{Introduction}\label{intro}

The current  leading framework for a quantum theory of gravity is sypersymmetric string theory,   in particular the five  well-established  and self-consistent   ten-dimensional superstring theories. 
  Due to the discovery of  string dualities, these  five distinct theories are now recognized to be different manifestations of a single underlying superstring theory  in eleven dimensions,  see  for example \cite{Green}.  This is the so-called {\it M-theory}, whose core  idea is that under  certain assumptions about the geometry of the underlying spacetime,   all
the existing superstring theories are  ``equivalent'' (in terms of T-duality or other kinds of dualities), see   \cite{Duff2, Witten, Green2, Tanii}  for more details and references.

 Nowadays, both ten-dimensional supergravity theories and their unique  eleven-dimensional analogue,   have been established as   the low-energy effective theories of  string theories and in particular of M-theory.
 As low-energy effective field theories, most supersymmetric supergravity theories supply us with a suitable framework for studying and understanding  string theory. For example,  all supergravity actions  consist both of bosonic and fermionic fields. The fermionic data is related with  matter degrees of freedom,  while the supersymmetries transformations relate the bosonic and fermionic fields each other.  A primary goal is always to understand the solution spaces of consistent backgrounds.      On the other hand, the supersymmetries for a given bosonic supergravity background   are determined  via a generalized Killing spinor equation.  The spinorial solutions    of this equation, the corresponding Killing superalgebras and the holonomy of the associated supercovariant derivative, are concepts playing a dynamic  role  towards to a  classification of supersymmetric   backgrounds. Hence, nowadays there is  a wide class of well-established solutions, as branes, monopoles, special kinds of Lorentzian manifolds or metrics of special holonomy (see for example \cite{Papa, Fppwav, FOF, FOFP}). However,  in general the classification of   (bosonic) supergravity backgrounds   lacks of a systematic investigation, and  most known results are    based for instance on  trivial flux forms, or  on  flux forms related with  $G$-structures or other special constructions  (see  \cite{Papa, Fppwav,   Sfetsos, Mee, House, Fig2, Fig3} and  the references therein).   In particular, the full classification of all supergravity backgrounds is a hard open problem.

Here, we consider eleven-dimensional supergravity (11d-supergravity for short), where the set of bosonic fields   consists of a  Lorentzian metric $h$  and the flux form $\mathsf{F}$, i.e. a  4-form gauge field $\mathsf{F}\in\Omega^{4}(\mathsf{X})$, globally defined on an eleven-dimensional spacetime $\mathsf{X}^{1, 10}$,  and satisfying  the following field equations  (cf. \cite{Cre, FOF, ACT})
\begin{eqnarray*}
\dd \mathsf{F}&=&0\,, \\
\dd \star \mathsf{F} &=&\frac{1}{2}\mathsf{F} \wedge \mathsf{F}\,, \\
\Ric^{h}(X,Y)&=&-\frac{1}{2} \langle X\lrcorner \mathsf{F}, Y\lrcorner \mathsf{F}\rangle_{h} +\frac{1}{6}h(X,Y) \|\mathsf{F}\|^2_{h}\,.
\end{eqnarray*}
The second equation is the   {\it Maxwell equation} while the third one is the so-called  {\it supergravity Einstein equation}.
Triples $(\mathsf{X}^{1, 10}, h, \mathsf{F})$ solving this system of equations are called {\it  bosonic supergravity backgrounds}.  11d-supergravity was introduced in \cite{Cre} and we refer to \cite{FOF, FOFP, Fig2} for further details. 
In this work, motivated by the previous remark,  we provide a  methodological  description of the bosonic supergravity field equations on eleven-dimensional Lorentzian manifolds, which are warped products of   five-dimensional Riemannian manifolds and  six-dimensional Lorentzian manifolds, i.e.
\[
\mathsf{X} = M ^5 \times_{f} \wi M ^{1, 5}\,,\quad h = g +   f^2\tilde g\,,
\]
for some non-trivial smooth function $f  : M\to\bb{R}$. The  tangent space $V=T_{x}\mathsf{X}$  of $\mathsf{X}^{1, 10}$ at a point $x=(p, q)\in\mathsf{X}$ decomposes  into a direct sum  of the five-dimensional (Euclidean) tangent space  $E=T_{p}M$ of $M$ and  the six-dimensional (Minkowski) tangent space $L=T_{q}\wi{M}$  of $\wi{M}$,  
\[
V^{1, 10}=E^{5}\oplus L^{1, 5}\,.
\]
Adapted to this  splitting,  there is an  orthogonal decomposition of the space $\Omega ^4 (\mathsf{X} )$ of  4-forms  on $\mathsf{X}$, which in the linear setting is given by
\[
\Lambda^{4}V=\Lambda ^4 \bb{R}  ^{ 1,10 } =\Lambda ^4 L \bigoplus (\Lambda ^3 L \wedge \Lambda^{1}E) \bigoplus (\Lambda ^2 L \wedge \Lambda ^2 E) \bigoplus (\Lambda^{1} L \wedge \Lambda ^3   E)\bigoplus  \Lambda ^4 E \,.
\]
Based on this decomposition,  we consider  4-forms of the type  
\[
\mathsf{F}:=  \tilde \al +   \tilde \be \wedge \nu +  \tilde \gamma \wedge \delta  +  \tilde \varpi  \wedge \eps  + \theta \,,
\]
where   $ \tilde \al \in \Omega ^4 (\wi{M} )$, $ \tilde \be \in  \Omega ^3 (\wi{ M} )$, $ \tilde  \gamma  \in \Omega  ^2 ( \wi{ M} ) $, $ \tilde \varpi \in \Omega ^1  (\wi{ M}) $, $ \nu \in \Omega ^1 (M ) $, $ \delta \in \Omega ^2 (M) $, $ \eps \in \Omega ^3 (M) $, and $\theta \in \Omega ^4 (M) $. Such 4-forms are {\it not} the most general 4-forms on $\mathsf{X}^{1, 10}$ (even if  each  factor  is multiplied by a  smooth function).  However, we will show that $\mathsf{F}$  serves as a nice representative which  induces several  cases providing bosonic supergravity backgrounds. 

Indeed, based on $\mathsf{F}$ we  describe  the constraints which appear due to  the closedness   condition, the Maxwell equation and the supergravity Einstein equation. As we expected, in the general case  the  induced system   is very complicated and not so meaningful, see for example Proposition \ref{MAX} for the Maxwell equation and Propositions \ref{EinHH}, \ref{EinVV} for the  supergravity Einstein equation.   
With aim to overcome this difficulty  and obtain a more tractable version, it is natural to treat the several special cases which  are induced by the flux form  $\mathsf{F}$.   Then,  the constraints which occur by the Maxwell equation in combination with the closedness  condition are simplified (see Proposition \ref{spck}), and the same applies for the supergravity Einstein field equation  (see Corollaries \ref{spcasealphaa}-\ref{spcasesumright1}).  

This procedure  finally leads to   the description of several new decomposable, non-symmetric, bosonic supergravity backgrounds. In fact, initially, we present general solutions given by untwisted products   of five-dimensional Ricci-flat Riemannian manifolds and six-dimensional Lorentzian manifolds, which satisfy some extra geometric conditions, e.g. their Ricci tensor  verifies  a generalized Einstein equation given  in terms  of a stress-energy tensor  associated to a non-trivial null 3- or 4-form, or a  parallel null 1-form (see Theorems \ref{thmnew1}, \ref{newthem2}, \ref{null13form} and see also Theorem \ref{spckckck} for a more complicated case).  When the flux form is  related to  a parallel null 1-form   we show in addition  that the  bosonic supergravity background $(\mathsf{X}^{1, 10}, h)$ must be  totally Ricci isotropic.    
Recall that  a semi-Riemannian manifold  $(\mathsf{X}, h)$ is called  {\it totally Ricci isotropic},  if the image of the Ricci endomorphism $\ric^{h} : T\mathsf{X}\to T\mathsf{X}$ is made up of isotropic vector fields.  Such geometric structures are very important  in holonomy theory of Lorentzian manifolds (cf. \cite{Leis, Gal, LorEin}) and they yield relevant applications in supergravity theory too (see also \cite{Fppwav}). 

In order to illustrate our general statements  and provide a suitable theoretical framework  inducing explicit bosonic solutions of 11d-supergravity, a  key challenge is the use of  special classes of  (Lorentzian) {\it Walker manifolds}, in particular  {\it Ricci-isotropic Walker manifolds}.  Walker manifolds are pseudo-Riemannian manifolds admitting a parallel null distribution  and they  serve as a powerful tool of constructing interesting indefinite metrics which exhibit various geometric aspects. Hence they  play a distinguished role  both in geometry and physics, see \cite{Broz} for several examples and see also   \cite{Gib, Gal,  LorEin,  Gal1} for a contribution of such manifolds in  Lorerntzian holonomy theory and Einstein theory. 

Here we focus on {\it Ricci-isotropic Walker manifolds}, where the notation  refers  to the property that the image of the Ricci endomorphism  related to the Walker metric is totally null.   Particular examples of such  manifolds are the  {\it pp-waves}, which can be also viewed as  {\it Brinkman spaces}.   
 As manifolds for which the Einstein field equations are linear, pp-waves   induce   a wealth of strong applications in    general relativity.  
 On the other hand,  as manifolds admitting a large number of parallel spinors,  pp-waves  also contribute  in  holonomy theory   and in supersymmetric  string  and supergravity theories. For example, the type IIB string theory along pp-wave backgrounds,  reduces to free massive theory in the light-cone
gauge and becomes exact solvable, see   \cite{Metsa}.    11d-supergravity also admits  pp-wave solutions  \cite{Kow, Hul, G}, which  provide simple but non-trivial backgrounds for studying  the AdS/CFT correspondence and maximal supersymmetry (see \cite{Hul, FP, Blau, G, Ber,   Cv, Baum, Ke, Fig2,  Leis, Gal} for more details in all these directions).  

Our strategy   to   get benefit of   Ricci-isotropic Walker manifolds is based on   using   mixed flux forms, given  in terms of their local coordinates in six dimensions  and the local coordinates of some Ricci-flat five-dimensional Riemannian manifold. We should mention that   flux forms induced by pp-waves have been applied to  several supergravity theories, e.g. five-dimensional pp-wave solutions with unbroken supersymmetry,    chiral six-dimensional supergravity theories of   type (1,0) and (2,0) \cite{Fig2, CFS, Mee}, or  11d-supergravity  \cite{Fppwav, Blau, Cv, FMP}.  However,  our results are illustrated  in terms of  more general Ricci-isotropic Walker metrics than pp-wave metrics  (see Corollaries \ref{sol1}, \ref{sol2}, \ref{sol3},  Remark \ref{similar} and Proposition \ref{sol4}).  Moreover,  the stated  bosonic supergravity backgrounds   are {\it non-symmetric} (compare with \cite{Fig3}).  At this point we   conjecture that,  to the given constructions,  different  approaches could be applied  (adapted  for example  to generic Walker metrics, or to the  stated results for warped metrics).  
   
   Regarding the Riemannian part  of the given decomposable backgrounds,  the supergravity Einstein equation   establishes  a link with  manifolds of special holonomy.  In particular, at least for the direct product construction and by considering null forms for the Lorentzian part, we   show that  the supergravity Einstein equation yields the constraint of Ricci-flatness for $(M^{5}, g)$. Consequently, and as we pronounced above, the stated decomposable   solutions are also based   on five-dimensional   Ricci-flat manifolds, given for instance in terms of products of   a Ricci-flat K\"ahler manifold with the real line  (see for example Corollaries \ref{sol1}, \ref{sol2}, \ref{sol3} and Proposition \ref{sol4}). 
      
    The structure of the paper  is given as follows.  In the first section we set  up our notation, relevant to our subsequent computations. In the second section we  fix  warped products of the form $M^{5}\times_{f}\wi{M}^{1, 5}$  and  describe the system of constraints, induced by the closedness  condition and Maxwell equation simultaneously, always with respect to  the  adapted 4-form $\mathsf{F}$.   The supergravity Einstein equation is discussed in the third section, where we also analyse the   special cases induced by  $\mathsf{F}$.  The fourth section contains  the explicit solutions and it is characterized by a potpourri of constructions and examples, mainly related with   Ricci-isotropic Walker metrics and null forms.    In the final section    we include   a few constructions, adapted to the interesting case where  the flux form $\mathsf{F}$ depends only on the Riemannian factor, and we pose some open problems. 
    
    Since we   present new bosonic non-symmetric supergravity backgrounds, it worths to analyse  the  corresponding spinoral equation and determine the corresponding  supersymmetries, if any.  Results in this direction will be presented  in a forthcoming  work.

\section{Preliminaries}
\subsection{Conventions}
Let  $(\mathsf{X}, h)$ be  a $n$-dimensional pseudo-Riemannian manifold of signature $(r, s)$.
As a matter of terminology, in this paper a Riemannian manifold  is assumed to have  signature $(0,s )$, which means that $h$ is a negative definite  metric.
Then, a unit vector $X$ on a Riemannian manifold  satisfies $h(X, X)=-1$. Also, we consider  Lorentzian manifolds as pesudo-Riemannian manifolds with  signature $(1, n-1 )$ (``mostly minus'').  A vector $Y$ is then time-like if $h(Y, Y)>0$, null if $h(Y, Y)=0$ and spacelike if $h(Y, Y)<0$.

We shall  denote by $\langle \ , \ \rangle_{h}$ the inner product induced by  $h$    on the space of differential forms. For $\alpha , \beta \in \Omega ^p (\mathsf{X}) $, it is given by
\[
\langle \al , \be\rangle_{h} 
= \frac{1}{p!}\sum_{\substack{\mu_1, \ldots,  \mu_p \\ \nu_1, \ldots, \nu_p}} \al _{ \mu _1 \ldots \mu _p }\be _{ \nu _1 \ldots \nu _p } h^{ \mu _1 \nu _1 } \cdots h ^{ \mu _p \nu _p }\,,
\]
and then we may  set $h (\alpha , \beta ) =p! \langle \alpha , \beta \rangle _h$.
The Hodge-star  $\star : \Omega^{p}(\mathsf{X})\to \Omega^{n-p}(\mathsf{X})$ is defined  by $\al \wedge \star\be = \langle \al , \be\rangle_{h}\mathrm{vol} _{\mathsf{X}}$, 
where $\vol_{\mathsf{X}}$   denotes the volume form of $(\mathsf{X}, h)$. 
\subsection{Warped products}
Let  $(\wi M, \tilde g ) $ and $(M,g) $ be  pseudo-Riemannian manifolds of dimensions $p $, $q$, and let  $\tilde s$, $ s $ be the number of negative eigenvalues of $\tilde g $, $g$, respectively.  Given a non-vanishing smooth function $f\in C^{\infty}(M)$, the  warped product   $ \mathsf{X}= M \times _f \wi M $ is the product manifold $M\times \wi{M}$ equipped with 
the pseudo-Riemannian  metric     $h = g + f ^2 \tilde g$. Usually,  $M$ is called the base, $\wi M$ the fiber and $f$ the warping function.  The volume form  of $(\mathsf{X}, h)$  satisfies $\mathrm{vol} _{ \mathsf{X} } =  f ^{p} \, \mathrm{vol} _{ \wi{M}} \wedge \mathrm{vol} _M$.  

Next we shall denote by   $ \star $, $\tilde * $, $ * $  the Hodge star  operator on $(\mathsf{X},h) $, $ (\wi M , \tilde g )$, and $(M,  g )$, respectively.  Some useful properties of  $\langle \ , \ \rangle_{h}$ and  $\star$    are given as follows  (a proof is direct, see also \cite[Lem.~2.1]{ACT} for $f=1$).\begin{lemma} \label{stvarious} 
Let $\tilde \al  \in \Omega ^{\tilde k} (\wi M) $, $ \be \in \Omega ^k (M) $. Then  
\[
\langle \tilde \al  \wedge \be , \tilde \al \wedge \be  \rangle_{h}= f ^{ -2 \tilde k } \langle \tilde \al , \tilde \al\rangle_{\tilde g}\langle \be , \be \rangle_{g}\,,\quad  \star (\tilde \al \wedge \be ) = (-1 )^{ k (p - \tilde k )}f ^{p-2 \tilde k } \  \tilde * \tilde \al \wedge * \be\,.
\]
\end{lemma}

\section{The Maxwell equation on warped products}
\subsection{Warped products  of the form $\mathsf{X} = M ^5 \times _f \wi M ^{1, 5}$}
In the remainder of this section we will study warped products of the form  $(\mathsf{X}^{1, 10} = M ^5 \times _f \wi M ^{1, 5}, h = g + f ^2 \tilde g)$, 
where $(\wi M^{1, 5}, \tilde g )$ is an six-dimensional oriented  Lorentzian manifold,  $(M^{5}, g)$ is a five-dimensional oriented Riemannian manifold and $f \in C ^\infty (M^5)$.   $\mathsf{X}^{1, 10}$ is oriented, with volume form given by $\vol_{\mathsf{X}}=f^{6}\vol_{M}\wedge\vol_{\wi{M}}$. We are    interested in closed 4-forms $\mathsf{F}\in \Omega ^4_{\rm cl} (\mathsf{X}^{1, 10}) $ on $\mathsf{X}^{1, 10}$.  The tangent space $V:=T_{x}\mathsf{X}$  of $\mathsf{X}^{1, 10}$ at a point $x=(p,q)\in\mathsf{X}$ decomposes as  $V\simeq \bb{R}  ^{1,10 }  = E ^5 \oplus L ^{1,5}$, and then motivated by the decomposition of  $\Lambda^{4}V$ described in introduction,
we consider   4-forms of the type
\begin{equation}\label{4formdeco} 
\mathsf{F} =  \tilde \al +   \tilde \be \wedge \nu +  \tilde \gamma \wedge \delta  +  \tilde \varpi  \wedge \eps  + \theta\,,
\end{equation} 
where $\tilde \al \in \Omega ^4 (\wi{M} )$, $\tilde \be \in  \Omega ^3 (\wi{ M} )$,  $\tilde  \gamma  \in \Omega  ^2 ( \wi{ M} )$, $\tilde \varpi \in \Omega ^1  (\wi{ M})$, and  $\nu \in \Omega ^1 (M )$,  $\delta \in \Omega ^2 (M)$, $\eps \in \Omega ^3 (M)$, $\theta \in \Omega ^4 (M)$, respectively.
In these terms  it is easy to see that
\begin{lemma} \label{CLOS}
The closedness of the  4-form  $\mathsf{F}$ defined by (\ref{4formdeco})  is equivalent to the following system of equations 
\[
\mathrm{d} \tilde \al=\mathrm{d} \theta=\mathrm{d} \tilde \be=\mathrm{d} \eps =0\,,\quad 
\mathrm{d} \tilde  \gamma \wedge \delta -  \tilde \be \wedge \mathrm{d} \nu =0\,, \quad 
 \tilde \gamma \wedge \mathrm{d} \delta + \mathrm{d}  \tilde \varpi \wedge \eps =0\,.
\]
\end{lemma} 
\noindent About the Maxwell equation, we apply  Lemma \ref{stvarious}   to obtain that 
 \begin{lemma} \label{dFF}
The 4-form  $\mathsf{F}$ defined by (\ref{4formdeco}) satisfies
\begin{eqnarray*}
\frac{1}{2} \mathsf{F} \wedge \mathsf{F} &=&  
 \tilde \al \wedge  \tilde \gamma \wedge \delta +   \tilde \al \wedge  \tilde \varpi \wedge \eps +  \tilde \be \wedge  \tilde \gamma \wedge \delta \wedge \nu  +  \tilde \al \wedge \theta +  \tilde \be \wedge  \tilde \varpi \wedge \eps \wedge \nu  \\
 && + \frac{1}{2}  \tilde \gamma \wedge  \tilde \gamma \wedge \delta \wedge \delta    +  \tilde \be \wedge \theta \wedge \nu +  \tilde \gamma \wedge  \tilde \varpi \wedge \eps \wedge \delta\,.
\end{eqnarray*}
Moreover, $\star \mathsf{F} =  f ^{ -2 } \tilde * \tilde\al \wedge \mathrm{vol} _M - \tilde * \tilde \be \wedge * \nu + f ^2 \tilde * \tilde \gamma \wedge * \delta - f ^4 \tilde * \tilde \varpi \wedge * \eps + f ^6 * \theta \wedge \mathrm{vol} _{ \wi M }$ 
and hence
\begin{eqnarray*}
\mathrm{d} \star \mathsf{F} &=&
\frac{1}{f ^2 } \mathrm{d} \tilde * \tilde \al \wedge \mathrm{vol} _M + \tilde * \tilde \be \wedge \mathrm{d} * \nu + 2 \tilde * \tilde \gamma \wedge f \mathrm{d} f \wedge * \delta  +  \tilde * \tilde \gamma \wedge f ^2 \mathrm{d} * \delta  - \mathrm{d} \tilde * \tilde \be \wedge * \nu \\ 
&& + \mathrm{d} \tilde * \tilde \gamma \wedge f ^2 * \delta  + f ^4 \tilde *  \tilde \varpi \wedge \mathrm{d} * \eps  + \tilde * \tilde \varpi \wedge 4 f ^3 \mathrm{d} f \wedge * \eps  \\
&& + \mathrm{vol} _{ \wi M} \wedge (6 f ^5 \mathrm{d} f \wedge * \theta + f ^6 \mathrm{d} * \theta ) - \mathrm{d} \tilde * \tilde \varpi \wedge f ^4 * \eps\,.
\end{eqnarray*}
\end{lemma} 
\noindent   Thus,  by  choosing in $\frac{1}{2} \mathsf{F} \wedge \mathsf{F}$ and $\mathrm{d} \star \mathsf{F}$ the terms of the same type,  we deduce that 
\begin{prop}\label{MAX}
On the Lorentzian manifold $(\mathsf{X}^{1, 10} = M ^5 \times _f \wi M ^6, h=g + f ^2 \tilde g)$ and for  the  4-form   $\mathsf{F}$  given by (\ref{4formdeco}), the   Maxwell equation is equivalent to the following system of equations:
\begin{eqnarray*}
\frac{1}{f ^2 } \mathrm{d} \tilde * \tilde \al \wedge \mathrm{vol} _M+ \tilde * \tilde \be \wedge \mathrm{d} * \nu &=& \tilde \be \wedge \theta \wedge \nu +  \tilde \gamma \wedge  \tilde \varpi \wedge \eps \wedge \delta \,, \\
2 \tilde * \tilde \gamma \wedge f \mathrm{d} f \wedge * \delta  +  \tilde * \tilde \gamma \wedge f ^2 \mathrm{d} * \delta  - \dd \tilde * \tilde \be \wedge * \nu &=& 
 \tilde \al \wedge \theta +  \tilde \be \wedge  \tilde \varpi \wedge \eps \wedge \nu  + \frac{1}{2}  \tilde \gamma \wedge  \tilde \gamma \wedge \delta \wedge \delta \,, \\
  \mathrm{d} \tilde * \tilde \gamma \wedge f ^2 * \delta  + f ^4 \tilde *  \tilde \varpi \wedge \mathrm{d} * \eps  + \tilde * \tilde \varpi \wedge 4 f ^3 \mathrm{d} f \wedge * \eps &=&   \tilde \al \wedge  \tilde \varpi \wedge \eps +  \tilde \be \wedge  \tilde \gamma \wedge \delta \wedge \nu \,, \\
  \mathrm{vol} _{ \wi M} \wedge (6 f ^5 \mathrm{d} f \wedge * \theta + f ^6 \mathrm{d} * \theta ) - \mathrm{d} \tilde * \tilde \varpi \wedge f ^4 * \eps &=& \tilde \al \wedge  \tilde \gamma \wedge \delta\,.
 \end{eqnarray*}
 \end{prop}

\subsection{Special cases}\label{specialc} 
The system of equations given by Lemma \ref{CLOS} and Proposition \ref{MAX} is in general  not tractable. It  becomes more meaningful when one considers simplifications of the flux form $\mathsf{F}$.  Below we list  the most interesting cases induced by $\mathsf{F}$   and    describe the system which generates the closedness condition, together with the Maxwell equation. 
\begin{example}\label{spcaseall} 
\textnormal{Assume that $\mathsf{F} =  \tilde \al +   \tilde \be \wedge \nu +   \tilde \varpi  \wedge \eps  + \theta$.  	This case is still  very complicated. For the record,  notice  by Lemma \ref{CLOS} the equation ${\rm d}\mathsf{F}=0$ imposes the closure of all the forms  which take part in the definition of $\mathsf{F}$. Also, by Proposition \ref{MAX} and by comparing the corresponding types of differential forms, we see that the Maxwell equations reads by  
\[
\begin{split}
( \tilde  3 , 5 )\quad\Rightarrow\quad  & \frac{1}{f ^2 } \mathrm{d} \tilde * \tilde \al \wedge \mathrm{vol} _M + \tilde * \tilde \be \wedge \mathrm{d} * \varpi =\tilde \be \wedge \varpi \wedge \theta\,,\\
( \tilde 4, 4 )\quad\Rightarrow \quad & - \mathrm{d} \tilde * \tilde \be \wedge * \varpi =\tilde \al \wedge \theta - \tilde \be  \wedge \tilde \varpi \wedge \nu \wedge \eps \,,\\
( \tilde 5, 3 )\quad\Rightarrow \quad &  \tilde * \tilde \varpi \wedge \mathrm{d} \left( f ^4   * \eps  \right)  =\tilde \al \wedge \tilde \varpi \wedge * \eps \,,\\
( \tilde 6, 2 )\quad\Rightarrow \quad & \mathrm{vol} _{ \wi M} \wedge \mathrm{d} ( f ^6   * \theta ) - f ^4  \mathrm{d} \tilde * \tilde \varpi \wedge * \eps =0\,.
\end{split}
\] }
\end{example}

\begin{prop}\label{spck}
Consider the Lorentzian manifold $(\mathsf{X}^{1, 10} = M ^5 \times _f \wi M ^6, h=g+f^{2}\tilde{g})$. Then\\
(1) The 4-form $\mathsf{F}\in\Omega^{4}(\mathsf{X})$ defined by 
\begin{equation} \label{spcasealpha} 
\mathsf{F} =\tilde \al 
\end{equation} 
satisfies the Maxwell equation and the closedness  condition  if and only if  
\[
\mathrm{d} \tilde \al = \mathrm{d} \tilde * \tilde \al =0\,.
\]
(2) The 4-form $\mathsf{F}\in\Omega^{4}(\mathsf{X})$ defined by 
\begin{equation} \label{spcasebeta} 
\mathsf{F} =\tilde \be \wedge \nu  
\end{equation} 
satisfies the Maxwell equation and the closedness condition  if and only if 
\[
\mathrm{d} \tilde \be  = \mathrm{d} \tilde * \tilde \be  =0\,, \quad \mathrm{d}  \nu   = \mathrm{d}  * \nu   =0\,.
\]
(3)  The 4-form $\mathsf{F}\in\Omega^{4}(\mathsf{X})$ defined by 
\begin{equation} \label{spcasegamma} 
\mathsf{F} =\tilde \gamma  \wedge \delta   
\end{equation} 
satisfies the Maxwell equation and the closedness condition  if and only if 
\[
\mathrm{d} \tilde \gamma   = \mathrm{d} \tilde * \tilde \gamma   =0\,, \quad
 \tilde * \tilde \gamma  \wedge \mathrm{d} (f ^2 * \delta )=\frac{\tilde \gamma \wedge \tilde \gamma \wedge \delta \wedge \delta }{2}\,.
\]
The 2nd condition is equivalent to $\tilde * \tilde \gamma = \con \tilde \gamma \wedge \tilde \gamma$ and $\mathrm{d} (f ^2 * \delta )=\frac{\delta \wedge \delta }{2\con}$, 
for some non-zero constant $\con$.\\
(4)  The 4-form $\mathsf{F}\in\Omega^{4}(\mathsf{X})$ defined by 
\begin{equation} \label{spcasevarpi} 
\mathsf{F} =\tilde \varpi  \wedge \eps  
\end{equation} 
satisfies the Maxwell equation and the closedness condition  if and only if 
\[
\mathrm{d} \tilde \varpi   = \mathrm{d} \tilde * \tilde \varpi   =0\,, \quad \mathrm{d} \eps     = \mathrm{d}   ( f ^4  * \eps)=0\,.
\]
(5)  The 4-form $\mathsf{F}\in\Omega^{4}(\mathsf{X})$ defined by 
\begin{equation} \label{spcasetheta} 
\mathsf{F} =\theta 
\end{equation} 
satisfies the Maxwell equation and the closedness condition  if and only if 
\[
\mathrm{d} \theta  = \mathrm{d} (f ^6  *\theta)  =0\,.
\]
(6)  The 4-form $\mathsf{F}\in\Omega^{4}(\mathsf{X})$ defined by 
\begin{equation}\label{spcasesumleft} 
\mathsf{F} =\tilde \al + \tilde \be \wedge \nu 
\end{equation} 
satisfies the Maxwell equation and the closedness condition  if and only if 
\[
\mathrm{d}\tilde  \al =\mathrm{d} \tilde \be=\mathrm{d} \nu = \mathrm{d} \tilde * \tilde \be  =0\,,  \quad 
\mathrm{d} * \nu = \frac{\con}{f ^2 }  \mathrm{vol} _M\,, \quad \tilde * \tilde \be =- \frac{1}{\con} \mathrm{d} \tilde * \tilde \al 
\]
for some non-zero constant $\con$.\\
(7)  The 4-form $\mathsf{F}\in\Omega^{4}(\mathsf{X})$ defined by 
\begin{equation}\label{spcasesumright} 
\mathsf{F} =\tilde \varpi \wedge \eps  + \theta
\end{equation} 
satisfies the Maxwell equation and the closedness condition  if and only if 
\[
\mathrm{d} \theta =\mathrm{d} \eps =\mathrm{d} \tilde \varpi = \mathrm{d}  (f ^4 * \eps )   =0\,, \quad 
\mathrm{d} \tilde * \tilde \varpi  = \con  \mathrm{vol} _{\wi M} \,, \quad  *  \eps  =  \frac{1}{\con} \mathrm{d} ( f ^6  *  \theta )\,,
\]
for some non-zero constant $\con$.\\
(8)  The 4-form $\mathsf{F}\in\Omega^{4}(\mathsf{X})$ defined by 
\begin{equation}\label{alphtheta} 
\mathsf{F} =\tilde\al  + \theta
\end{equation} 
satisfies the Maxwell equation and the closedness condition  if and only if either $\tilde\al =0 $  or $\theta =0 $, and moreover $\theta$, (respectively $\tilde\al$) satisfies the  conditions described in  case (5) (respectively, case (1)).   \\
(9) The examination of the 4-form $\mathsf{F}\in\Omega^{4}(\mathsf{X})$ defined by 
\begin{equation}\label{bnwe} 
\mathsf{F} = \tilde \be \wedge \nu + \tilde \varpi \wedge \eps 
\end{equation} 
reduces to the case $\mathsf{F} =\tilde \varpi \wedge \eps$,  i.e.  $\nu=0$.
\end{prop}
\begin{proof}
When  $\mathsf{F} =\tilde\al + \theta$ we get the relation $\tilde\al \wedge \theta =0 $, which yields our claim for this particular case.  Also, regarding the  case where $\mathsf{F}$ is defined by  (\ref{bnwe}), we obtain the condition $* \nu =\nu \wedge \eps$, hence $\nu =0 $.
\end{proof}

\section{The supergravity Einstein  equation on warped  products}\label{SUGRAE}
Consider the Lorentzian manifold $(\mathsf{X}^{1, 10} = M ^5 \times _f \wi M ^6, h=g+f^{2}\tilde{g})$. In this section we analyse  the supergravity Einstein equation, i.e. 
\begin{equation}
\label{einsteqq} 
\Ric^{h}(X , Y) = - \frac{1}{2} \langle  i _X \mathsf{F} , i _Y \mathsf{F}\rangle_{h}  + \frac{1}{6} h (X , Y ) \| \mathsf{F} \| ^2 _h \,,
\end{equation} 
for any vector field $X, Y\in\fr{X}(\mathsf{X})$.   It  is useful to recall   the Ricci tensor   of warped products. But first, based on Lemma \ref{stvarious} let us present the square norm $\| \mathsf{F} \|^2  _h$ of $\mathsf{F}$ with respect to the metric $h$.
\begin{lemma}
Let $\mathsf{F}$ be the 4-form on $\mathsf{X}^{1, 10} = M ^5 \times _f \wi M ^6$ defined by (\ref{4formdeco}).  Then, the square of the norm of  $\mathsf{F}$   with respect to the warped metric $h$ equals to
\begin{equation}
\label{f2norm} 
\begin{split}
\| \mathsf{F} \|^2  _h &
= \frac{\| \tilde \al \| ^2 _{ \tilde g }}{f ^8 }
+ \frac{\| \tilde \be  \| ^2 _{ \tilde g } \|\nu  \| ^2  _g }{f ^6}
+ \frac{\| \tilde \gamma   \| ^2 _{ \tilde g } \| \delta   \| ^2  _g }{f ^4}
+ \frac{\| \tilde \varpi   \| ^2 _{ \tilde g } \|\eps   \| ^2  _g }{f ^2}
+ \|\theta  \| ^2  _g \,.
\end{split} 
\end{equation} 
\end{lemma}

Consider a warped manifold  $(\mathsf{X} =M \times _f \wi M, h =g+ f ^2 \tilde g)$, for some non-vanishing  function $f \in C ^\infty (M)$. 
The projection $\pi : (\mathsf{X}, h)\to (M, g)$ onto $M$ becomes a semi-Riemannian submersion having integrable horizontal distribution and  the manifold $(\wi{M}, f^{2}(p)\tilde g)$ as fibre over  $p\in M$. 
\begin{prop}\label{wraON} \textnormal{(\cite[p.~211]{ON})}
On a warped  product  $(\mathsf{X} =M \times _f \wi M ,h)$ with $\mathrm{dim} (\wi M )>1 $ the Ricci tensors $\Ric^{h}, \Ric^{g}, \Ric^{\tilde{g}}$ of $(\mathsf{X}, h), (M, g)$ and $(\wi{M}, \tilde g)$, respectively, are related by
\[
\Ric^{h}( X , Y ) =  \Ric^{g}(X , Y ) - \frac{\dim_{\bb{R}}\wi M}{f} H ^f ( X , Y )\,,\quad  \Ric^{h} (\tilde X , \tilde Y )=\Ric^{\tilde{g}}(\tilde X , \tilde Y ) - h (\tilde X , \tilde Y ) \hat{f}\,,
\]
and $\Ric^{h}( X , \tilde Y )=0$, for any $X, Y  \in \mathfrak{X}  (M) $, $\tilde X , \tilde Y  \in \mathfrak{X}  (\wi M )$. Here,  $H ^f $ is the Hessian of $f \in C ^\infty (M) $ and  
\[
 \hat{f}:= \frac{ \Delta _h f}{f} + \left(\dim_{\bb{R}}\wi M-1 \right)  \frac{h(\mathrm{grad}f,\mathrm{grad}f) }{f ^2 }\,.
\]
\end{prop}
\noindent Of course, there is no ambiguity regarding  which metric one  uses to compute $H ^f $ or $\mathrm{grad} (f) $, as $f\in C ^\infty (M) $ and $ h |_M =g $\,.

\subsection{HV-part of the supergravity Einstein  equation}\label{HVpart}
We begin with the mixed part of the Einstein supergravity equation. In particular, for  any $X \in \mathfrak{X}  (M) $ and $ Y =\tilde Z \in \mathfrak{X}  (\wi M )$ by Proposition \ref{wraON} we deduce that the equation (\ref{einsteqq}) reduces to $\langle i_{X}\mathsf{F}, i_{\tilde Z}\mathsf{F}\rangle_{h}=0$,  or equivalently to  $h(i_{X}\mathsf{F}, i_{\tilde Z}\mathsf{F})=0$, which is expressed  by  
\[
0 = -g(X, \nu)h(\tilde \be , i _{ \tilde Z } \tilde \al) +h(\tilde \gamma , i _{ \tilde Z }\tilde \be)g(i _X \delta , \nu)
- h(\tilde \varpi , i _{ \tilde Z }\tilde \gamma)g(i _X \eps , \delta)  + h(\tilde Z , \tilde \varpi)g(i _X \theta , \eps)\,,
\]
where  for simplicity we maintain  the same notation for the 1-form $\nu$ and its dual vector field $\nu^{\sharp}$ defined by   $\nu(X)=g(X, \nu^{\sharp})$. This   finally reduces to 
\begin{equation}\label{useHV}
g(X, \nu ) \tilde g(  \tilde \be , i _{ \tilde Z } \tilde \al) + \tilde g( \tilde \varpi , i _{ \tilde Z }\tilde \gamma ) g(  i _X \eps , \delta ) 
=  \tilde g( \tilde \gamma , i _{ \tilde Z }\tilde \be ) g(  i _X \delta , \nu )  + \tilde g( \tilde Z , \tilde \varpi )g( i _X \theta , \eps )\,.
\end{equation} 
We shall call this condition   as the  {\it  HV-part of the supergravity  Einstein equation}. 
\begin{remark}\label{nicerem}\textnormal{Regardless of whether $\tilde \varpi $ is time-like, null or space-like, for the special cases (\ref{spcasealpha}) -- (\ref{spcasetheta}) described in Proposition  \ref{spck} it is easy to see that  the HV-part of the supergravity  Einstein equations is trivially satisfied. In the remaining  cases, we get the following constraints as direct consequences of (\ref{useHV}),
\begin{align} 
\label{mcab} 
\mathsf{F} =\tilde \al + \tilde \be \wedge \nu  \quad\Rightarrow \quad &
\tilde g (\tilde \be , i _{ \tilde Z }\tilde \al )=0, \ \forall \ \tilde Z  \in \mathfrak{X}  (\wi M)\,,\nonumber\\
\mathsf{F} =\tilde \varpi \wedge \eps + \theta\quad   \Rightarrow \quad &
 g ( i _{ X }\theta , \eps  )=0, \ \forall \ X \in \mathfrak{X}  (M)\,,\nonumber\\
\mathsf{F} =\tilde \al + \tilde \be \wedge \nu  + \tilde \varpi \wedge \eps + \theta \quad \Rightarrow \quad &
 g ( i _{ X }\theta , \eps  )=0 =\tilde g (\tilde \be , i _{ \tilde Z }\tilde \al )\,, \ \forall \ X \in \mathfrak{X}  (M) \,, \tilde Z \in \mathfrak{X}  (\wi M)\,.\nonumber
\end{align} }
\end{remark}

\subsection{HH-part of the supergravity Einstein  equation}\label{HHpart}
We examine now the restriction of the supergravity Einstein equation (\ref{einsteqq}) on the Riemannian component $(M^{5}, g)$ of the warped product $\mathsf{X}^{1, 10}=M^{5}\times_{f}\wi{M}^{1, 5}$.  We call this restriction  the {\it HH-part of the supergravity Einstein equation}.  For any $X, Y  \in \mathfrak{X}  (M) $ we have $
i _X \mathsf{F} 
=- \tilde \be  \nu  (X) + \tilde \gamma \wedge i _X \delta - \tilde \varpi \wedge i _X \eps + i _X \theta $, 
and hence 
\begin{equation}
\label{hixf} 
\langle i _X \mathsf{F} , i _Y \mathsf{F}\rangle_{h}
=\frac{1}{f ^6 }\nu (X) \nu (Y) \| \tilde \be \| ^2 _{ \tilde g }
+\frac{1}{f ^4 }\langle i _X \delta,  i _Y \delta \rangle_{g}  \| \tilde \gamma  \|^2  _{ \tilde g }
+\frac{1}{f ^2 }\langle i _X \eps,  i _Y \eps\rangle_{g} \| \tilde \varpi   \|^2  _{ \tilde g }
+ \langle i _X \theta , i _Y \theta\rangle_{g}\,.
\end{equation} 
Thus, a combination of  Proposition \ref{wraON} and relations (\ref{f2norm}), (\ref{einsteqq}) and (\ref{hixf}) gives that

\begin{prop}\label{EinHH}
For any $X, Y \in \mathfrak{X}  (M) $ the supergravity  equation (\ref{einsteqq}) reduces to
\begin{eqnarray*}
\Ric^{g}(X , Y ) &=& \frac{6}{f} H ^f ( X , Y) +  \frac{1}{6}\Big\{\frac{\|\tilde\al\|^{2}_{ \tilde g }}{f ^8}g(X, Y)+\frac{\| \tilde \be\| ^2_{ \tilde g}}{f ^6}\Big( \|\nu\| ^2 _{g}g(X, Y)-3\nu(X)\nu(Y)\Big)\\
&+&  \frac{\|\tilde\gamma\| ^2 _{ \tilde g }}{f^{4}}\Big( \| \delta   \| ^2  _{g}g(X, Y)-3\langle i_{X}\delta, i_{Y}\delta\rangle_{g}\Big) +\frac{\| \tilde \varpi\| ^2 _{\tilde g}}{f^{2}}\Big(\|\eps\| ^2_{g}g(X, Y)-3\langle i_{X}\eps, i_{Y}\eps\rangle_{g}\Big)\\
&+&\|\theta\|^{2}_{g}g(X, Y)-3\langle i_{X}\theta, i_{Y}\theta\rangle_{g}\Big\}\,.
\end{eqnarray*}
\end{prop}

\subsection{VV-part of the supergravity Einstein  equation}
Let us restrict now  the supergravity Einstein equation (\ref{einsteqq}) on the Lorentzian part $(\wi{M}^{1, 5}, f^{2}\tilde{g})$ of $(\mathsf{X}^{1, 10}, h=g+f^{2}\tilde g)$.  We shall refer to this restriction by the term  {\it VV-part of the supergravity Einstein  equation}. In this case, 
for any $\tilde X, \tilde Y  \in \mathfrak{X}(\wi M) $ we compute  $i _{ \tilde X} \mathsf{F} 
= i _{ \tilde X }\tilde \al + i _{ \tilde X }\tilde \be \wedge \nu  + i _{ \tilde X }\tilde \gamma \wedge \delta + \tilde \varpi (\tilde X )\eps$, 
and hence 
\begin{equation}
\label{hixftil} 
\langle i _{ \tilde X} \mathsf{F} , i _{\tilde Y} \mathsf{F}\rangle_{h}
=\frac{1}{f ^6 }  \langle i _{ \tilde X }\tilde \al , i _{ \tilde Y }\tilde \al\rangle_{\tilde g}
+\frac{1}{f ^4 }  \langle i _{ \tilde X} \tilde \be ,  i _{ \tilde Y} \tilde \be\rangle_{\tilde g}  \|  \nu   \|^2  _{ g }
+\frac{1}{f ^2 } \langle i _{\tilde X} \tilde \gamma ,  i _{ \tilde Y} \tilde \gamma\rangle_{\tilde g} \| \delta    \|^2  _{  g }
+ \tilde \varpi (\tilde X ) \tilde \varpi (\tilde Y )  \| \eps \| ^2  _g \,.
\end{equation} 
Since $h|_{\wi{M}}=f^{2} \, \tilde{g}$,  by  Proposition \ref{wraON}  and relations (\ref{f2norm}), (\ref{einsteqq}) and (\ref{hixftil}) one concludes that

\begin{prop}\label{EinVV}
For any $\tilde X, \tilde Y \in \mathfrak{X}  (\wi M)$ the supergravity Einstein equation (\ref{einsteqq}) reduces to 
\begin{eqnarray*}
\Ric^{\tilde g}(\tilde X , \tilde Y) &=&f^{2}\tilde{g} (\tilde X , \tilde Y) \hat{f}+  \frac{f^{2}}{6}\Big\{\frac{1}{f ^8}\Big(\|\tilde\al\|^{2}_{ \tilde g }\tilde{g}(\tilde X, \tilde Y)-3\langle i_{\tilde X}\tilde{\al}, i_{\tilde Y}\tilde{\al}\rangle_{\tilde g}\Big)\\
&+&\frac{\| \nu  \| ^2 _{g}}{f ^6}\Big(\|\tilde\be  \| ^2  _{\tilde g}\tilde{g}(\tilde X, \tilde Y)-3\langle i_{\tilde X}\tilde{\be}, i_{\tilde Y}\tilde{\be}\rangle_{\tilde g}\Big)+  \frac{\| \tilde \delta  \| ^2 _{ g }}{f^{4}}\Big( \| \tilde{\gamma}   \| ^2  _{\tilde g}\tilde{g}(\tilde X, \tilde Y)-3\langle i_{\tilde{X}}\tilde\gamma, i_{\tilde Y}\tilde\gamma\rangle_{\tilde g}\Big)\\
&+&\frac{\|\eps   \| ^2 _{g}}{f^{2}}\Big(\| \tilde \varpi \| ^2  _{\tilde g}\tilde{g}(\tilde X, \tilde Y)-3 \tilde \varpi(\tilde X) \tilde \varpi(\tilde Y)\Big)+\|\theta\|^{2}_{g}\tilde{g}(\tilde{X}, \tilde{Y})\Big\}\,.
\end{eqnarray*}
\end{prop}

\subsection{Special cases}\label{specialsugra}
Let us now present the  supergravity Einstein  equation  for the special cases indicated   in Proposition  \ref{spck}.
All the conclusions are based on Propositions \ref{EinHH} and \ref{EinVV},  and also on the constraints appearing (if any) of the  HV-part of the supergravity  Einstein equation.

\smallskip
\noindent	{\bf Case (\ref{spcasealpha})} Assume that $\mathsf{F} =\tilde \al $,  for some non-trivial 4-form $\tilde\al\in\Omega^{4}(\wi M)$.
\begin{corol}\label{spcasealphaa}
When the 4-form $\mathsf{F}$ is given by $\mathsf{F} =\tilde \al $, then the supergravity Einstein equation (\ref{einsteqq}) is equivalent to the following system of equations
\begin{eqnarray*}
\Ric^{g}(X , Y ) &=& \frac{6}{f} H ^f ( X , Y )+ \frac{\|\tilde\al\|^{2}_{ \tilde g }}{6f ^8}g(X, Y)\,,\\
\Ric^{\tilde g}(\tilde X , \tilde Y) &=&f^{2}\tilde{g} (\tilde X , \tilde Y) \hat{f}+ \frac{1}{6f^{6}}\Big(\|\tilde\al\|^{2}_{ \tilde g }\tilde{g}(\tilde X, \tilde Y)-3\langle i_{\tilde X}\tilde{\al}, i_{\tilde Y}\tilde{\al}\rangle_{\tilde g}\Big)\,,
\end{eqnarray*}
for any $X, Y \in \mathfrak{X}  (M) $ and $\tilde X, \tilde Y \in \mathfrak{X}  (\wi M)$.
\end{corol}

\smallskip
\noindent	{\bf Case (\ref{spcasebeta})}  Assume that $\mathsf{F}=\tilde{\be}\wedge\nu$ for some non-trivial  3-form $\tilde\be\in\Omega^{3}(\wi{M})$ and non-trivial  1-form $\nu\in\Omega^{1}(M)$.  
\begin{corol} \label{spcasebetaa} 
When the 4- form $\mathsf{F}$ is given by $\mathsf{F}=\tilde{\be}\wedge\nu$, then the supergravity Einstein equation (\ref{einsteqq})  is equivalent to the following system of equations
\begin{eqnarray*}
\Ric^{g}(X , Y ) &=& \frac{6}{f} H ^f ( X , Y )+\frac{\| \tilde \be  \| ^2 _{ \tilde g}}{6f ^6}\Big(\|\nu  \| ^2  _{g}g(X, Y)-3\nu(X)\nu(Y)\Big)\,,\\
\Ric^{\tilde g}(\tilde X , \tilde Y) &=&f^{2}\tilde{g} (\tilde X , \tilde Y) \hat{f}+ \frac{\| \nu  \| ^2 _{g}}{6f ^4}\Big( \|\tilde\be  \| ^2 _{\tilde g}\tilde{g}(\tilde X, \tilde Y)-3\langle i_{\tilde X}\tilde{\be}, i_{\tilde Y}\tilde{\be}\rangle_{\tilde{g}}\Big)\,,
\end{eqnarray*}
for any $X, Y \in \mathfrak{X}  (M) $ and   $\tilde X, \tilde Y \in \mathfrak{X}  (\wi M)$.
\end{corol}

\smallskip
\noindent	{\bf Case (\ref{spcasegamma})}  Assume that $\mathsf{F}=\tilde{\gamma}\wedge\delta$ for some non-trivial  $\tilde\gamma\in\Omega^{2}(\wi{M})$ and $\delta\in\Omega^{2}(M)$.  
\begin{corol}
When the 4-form $\mathsf{F}$ is given by $\mathsf{F}=\tilde{\gamma}\wedge\delta$,  then the supergravity Einstein equation (\ref{einsteqq}) is equivalent to the following system of equations
\begin{eqnarray*}
\Ric^{g}(X , Y ) &=& \frac{6}{f} H ^f ( X , Y )+ \frac{\| \tilde \gamma   \| ^2 _{ \tilde g }}{6f^{4}}\Big( \| \delta   \| ^2  _{g}g(X, Y)-3\langle i_{X}\delta, i_{Y}\delta\rangle_{g}\Big)\,, \\
\Ric^{\tilde g}(\tilde X , \tilde Y) &=&f^{2}\tilde{g} (\tilde X , \tilde Y) \hat{f}+ \frac{\| \tilde \delta  \| ^2 _{ g }}{6f^{2}}\Big( \| \tilde{\gamma}\| ^2  _{\tilde g}\tilde{g}(\tilde X, \tilde Y)-3\langle i_{\tilde{X}}\tilde\gamma, i_{\tilde Y}\tilde\gamma\rangle_{\tilde g}\Big)\,,
\end{eqnarray*}
for any $X, Y \in \mathfrak{X}  (M) $ and   $\tilde X, \tilde Y \in \mathfrak{X}  (\wi M)$.
\end{corol}

\smallskip
\noindent	{\bf Case (\ref{spcasevarpi}) }  Assume that $\mathsf{F} =\tilde \varpi  \wedge \eps$ for some non-trivial  $\tilde \varpi \in\Omega^{1}(\wi{M})$ and $\eps\in\Omega^{3}(M)$.  
\begin{corol}\label{spcasevarpii}
When the 4-form $\mathsf{F}$ is given by $\mathsf{F} =\tilde \varpi  \wedge \eps$, then the supergravity Einstein equation (\ref{einsteqq}) is equivalent to the following system of equations
\begin{eqnarray*}
\Ric^{g}(X , Y ) &=& \frac{6}{f} H ^f ( X , Y )+\frac{\| \tilde \varpi   \| ^2 _{ \tilde g}}{6f^{2}}\Big(\|\eps   \| ^2  _{g}g(X, Y)-3\langle i_{X}\eps, i_{Y}\eps\rangle_{g}\Big)\,,\\\
\Ric^{\tilde g}(\tilde X , \tilde Y) &=&f^{2}\tilde{g} (\tilde X , \tilde Y) \hat{f}+ \frac{\|\eps   \| ^2 _{g}}{6}\Big(\| \tilde \varpi \| ^2  _{\tilde g}\tilde{g}(\tilde X, \tilde Y)-3 \tilde \varpi(\tilde X) \tilde \varpi(\tilde Y)\Big)\,,
\end{eqnarray*}
for any $X, Y \in \mathfrak{X}  (M) $ and  $\tilde X, \tilde Y \in \mathfrak{X}  (\wi M)$.
\end{corol}

\smallskip
\noindent	{\bf Case (\ref{spcasetheta})}  Assume that $\mathsf{F}=\theta$ for some non-trivial  $\theta\in\Omega^{4}(M)$.   
\begin{corol}\label{spcasethetaa}
When the 4-form $\mathsf{F}$ is given by $\mathsf{F}=\theta$, then  the supergravity Einstein equation (\ref{einsteqq}) is equivalent to the following system of equations
\begin{eqnarray*}
\Ric^{g}(X , Y ) &=& \frac{6}{f} H ^f ( X , Y )+\frac{1}{6}\Big(\|\theta\|^{2}_{g}g(X, Y)-3\langle i_{X}\theta, i_{Y}\theta\rangle_{g}\Big)\,,\\
\Ric^{\tilde g}(\tilde X , \tilde Y) &=&f^{2}\tilde{g} (\tilde X , \tilde Y) \hat{f}+\frac{f^{2}}{6} \|\theta\|^{2}_{g}\tilde{g}(\tilde{X}, \tilde{Y})=\Big(f^{2}f^{\sharp}+\frac{f^{2}}{6}\|\theta\|^{2}_{g}\Big)\tilde{g}(\tilde{X}, \tilde{Y})\,,
\end{eqnarray*}
for any $X, Y \in \mathfrak{X}  (M) $ and  $\tilde X, \tilde Y \in \mathfrak{X}  (\wi M)$.
\end{corol}

\smallskip
\noindent	{\bf Case (\ref{spcasesumleft})}   Assume now that $\mathsf{F}=\tilde{\al}+\tilde{\be}\wedge\nu$ for some $\tilde{\al}\in\Omega^{4}(\wi{M}), \tilde{\be}\in\Omega^{3}(\wi{M})$ and $\nu\in\Omega^{1}(M)$.  
\begin{corol}\label{spcasesumleftt}
 When the 4-form $\mathsf{F}$ is given by $\mathsf{F}=\tilde{\al}+\tilde{\be}\wedge\nu$, then the supergravity Einstein equation (\ref{einsteqq})  is equivalent to the following system of equations
\begin{eqnarray*}
\Ric^{g}(X , Y ) &=& \frac{6}{f} H ^f ( X , Y )+\frac{\|\tilde\al\|^{2}_{ \tilde g }}{6f ^8}g(X, Y)+\frac{\| \tilde \be  \| ^2 _{ \tilde g}}{6f ^6}\Big(\|\nu  \| ^2  _{g}g(X, Y)-3\nu(X)\nu(Y)\Big)\,,\\
\Ric^{\tilde g}(\tilde X , \tilde Y) &=&f^{2}\tilde{g} (\tilde X , \tilde Y) \hat{f}+  \frac{1}{6f^{6}}\Big(\|\tilde\al\|^{2}_{ \tilde g }\tilde{g}(\tilde X, \tilde Y)-3\langle i_{\tilde X}\tilde{\al}, i_{\tilde Y}\tilde{\al}\rangle_{\tilde g}\Big)+\frac{\| \nu  \| ^2 _{g}}{6f ^4}\Big( \|\tilde\be  \| ^2 _{\tilde g}\tilde{g}(\tilde X, \tilde Y)-3\langle i_{\tilde X}\tilde{\be}, i_{\tilde Y}\tilde{\be}\rangle_{\tilde g}\Big)\,, \\
\tilde g (i _{ \tilde X }\tilde \al, \tilde \be)&=&0\,,
  \end{eqnarray*}
for any $X, Y \in \mathfrak{X}  (M) $ and  $\tilde X, \tilde Y \in \mathfrak{X}  (\wi M)$.
\end{corol}

\smallskip
\noindent	{\bf Case (\ref{spcasesumright})} Assume that $\mathsf{F} =\tilde \varpi \wedge \eps  + \theta  $ for some $\tilde\varpi\in\Omega^{1}(\wi{M})$, and $\eps\in\Omega^{3}(M), \theta\in\Omega^{4}(M)$. 
\begin{corol}\label{spcasesumright1}
When the 4-form $\mathsf{F}$ is given by $\mathsf{F} =\tilde \varpi \wedge \eps  + \theta  $, then the supergravity Einstein equation (\ref{einsteqq})  is equivalent to the following system of equations
\begin{eqnarray*}
\Ric^{g}(X , Y ) &=& \frac{6}{f} H ^f ( X , Y )+\frac{\| \tilde \varpi   \| ^2 _{ \tilde g}}{6f^{2}}\Big(\|\eps   \| ^2  _{g}g(X, Y)-3\langle i_{X}\eps, i_{Y}\eps\rangle_{g}\Big)+\frac{1}{6}\Big(\|\theta\|^{2}_{g}g(X, Y)-3\langle i_{X}\theta, i_{Y}\theta\rangle_{g}\Big)\,,\\
\Ric^{\tilde g}(\tilde X , \tilde Y) &=&f^{2}\tilde{g} (\tilde X , \tilde Y)\hat{f}+ \frac{\|\eps\| ^2_{g}}{6}\Big(\| \tilde \varpi \| ^2 _{\tilde g}\tilde{g}(\tilde X, \tilde Y)-3 \tilde \varpi(\tilde X) \tilde \varpi(\tilde Y)\Big)+\frac{f^{2}}{6}\|\theta\|^{2}_{g}\tilde{g}(\tilde{X}, \tilde{Y})\,,\\
g (i _{ X }\theta , \eps  )&=&0\,,
\end{eqnarray*}
for any $X, Y \in \mathfrak{X}  (M) $ and   $\tilde X, \tilde Y \in \mathfrak{X}  (\wi M)$.
\end{corol}

\section{Examples  of  supergravity backgrounds}
In this section we present  solutions of the bosonic supergravity equations on the   Lorentzian manifold $(\mathsf{X}^{1, 10} = M ^5 \times _f \wi M ^{1, 5}, h=g+f^{2}\tilde{g})$ for particular choices of the flux form $\mathsf{F}$, where in general $\mathsf{F}$ is given by (\ref{4formdeco}). It turns out that   a suitable Ansatz in order to simplify the whole problem and provide  decomposable   supergravity backgrounds, is related to  direct products $(f=1)$ and   null forms. In this way we provide some new supergravity backgrounds, which most of them are illustrated in terms of  Ricci-isotropic Walker metrics and in particular pp-waves.   However, we conjecture that  different explicit  bosonic solutions may exist, in terms for example of more general  Walker metrics or warped metrics, and  the underlying constraints for these directions must be established with respect to  the   equations presented in Proposition \ref{spck} and Section \ref{specialsugra}.  
For the Riemannian part, and at least for the direct product construction and based on null forms for  the Lorentzian part,  it turns out that the supergravity Einstein equation  implies  the corresponding Ricci flatness.  Thus, our solutions depend also on  {Ricci-flat manifolds}.

\subsection{Ricci-isotropic Walker manifolds and pp-waves}
For our constructions,  it is useful  to  refresh  basic properties of Lorentzian Walker manifolds (see also \cite{Gib, Gal, Gal1, Leis}).

Consider a simply-connected  Lorentzian manifold  $(\mathsf{X}^{1, n+1}, h)$   and identify the tangent space  $T_{x}\mathsf{X}^{1, n+1}$ with the Minkowski space $\bb{R}^{1, n+1}$. Since $\mathsf{X}^{1, n+1}$ is simply-connected, the holonomy group ${\rm Hol}_{x}(h)$ is determined by the holonomy algebra $\fr{g}\subseteq\fr{so}(1, n+1)$.   $\mathsf{X}$ is {\it locally indecomposable} if $\fr{g}$   does not preserve any     proper, non-degenerate subspace of $T_{x}\mathsf{X}^{1, n+1}$  (\cite{Wu}), i.e. $\fr{g}$ is a so-called {\it weakly irreducible} holonomy algebra . In this case, $\mathsf{X}$ is not locally the product of some pseudo-Riemannian manifolds.    If $\fr{g}$ is not equal to $\fr{so}(1, n+1)$, i.e. $\fr{g}$ is weak irreducible but not irreducible, then $\fr{g}$ preserves a degenerate subspace $\mathsf{U}\subset\bb{R}^{1, n+1}$ and also the  isotropic line $\ell:=\mathsf{U}\cap\mathsf{U}^{\perp}\subset\bb{R}^{1, n+1}$. It follows that  $\mathsf{X}$ admits a parallel distribution of isotropic lines.  It is well-known that  such manifolds admit local coordinates 
$(v, x^1, \ldots, x^n, u)$  such that
\[
h=2{\rm d}v{\rm d}u+\rho+2A{\rm d}u+H({\rm d}u)^{2}\,,
\]
where $\rho=\rho_{ij}(x^1, \ldots, x^{n}, u){\rm d}x^{i}{\rm d}x^{j}$ is a family of Riemannian metrics, $A=A_{i}(x^{1}, \ldots, x^{n}, u){\rm d}x^{i}$ is  1-form  independent of $v$ and $H=H(v, x^{1}, \ldots, x^{n}, u)$ is a local function on $\mathsf{X}$. Such local coordinates  are referred to as  {\it Walker coordinates}, due to \cite{Wal} and the ambient space  $\mathsf{X}$  can be assumed to be locally diffeomorphic  to $\bb{R}\times N\times\bb{R}$, for some family of  Riemannian manifolds $(N^n, \rho)$. The parallel distribution of isotropic lines $\ell$ is defined by the  vector field $\partial_{v}:=\partial/\partial v$.

Equivalently,  such Lorentzian manifolds are   characterized  by the fact that their  holonomy algebra $\fr{g}$ is contained in the maximal subalgebra $\fr{sim}(n)\subset\fr{so}(1, n+1)$  of $\fr{so}(1, n+1)$ preserving $\bb{R}\ell$, i.e.
\[
\fr{g}\subset\fr{sim}(n)=\left\{\begin{pmatrix} \mathsf{a} & X^{t} & 0 \\
0 & A & -X \\
0 & 0 & -\mathsf{a} \end{pmatrix} : \mathsf{a}\in\mathbb{R},  X\in\bb{R}^n, A\in\fr{so}(n)\right\}\simeq (\bb{R}\oplus\fr{so}(n))\ltimes\bb{R}^n\,.
\]
 Next we shall focus on Lorentzian Walker manifolds satisfying the following assumptions:
  \begin{itemize}
  \item[$\mathsf{(1)}$] $\partial_{v}H=0$ (and hence $\nabla^{h}\partial_{v}=0$); 
  \item[$\mathsf{(2)}$] $A_{i}(x^{1}, \ldots, x^{n}, u)=0$ for any $i$;
  \item[$\mathsf{(3)}$] $\rho$ is a family of Ricci-flat  Riemannian metrics. 
\end{itemize}
 When  $\partial_{v}H=0$, then one can show that  $\fr{g}=\fr{h}\ltimes\bb{R}^n$, where $\fr{h}$ is the holomomy algebra of  some Riemannian metric, see \cite{Gal}. 
  Note also that  the 1-form ${\rm d}u$ is  parallel and null, $h({\rm d}u, {\rm d}u)=0$, with $h({\rm d}u, {\rm d}v)=1$.

 Although the  assumptions  $\mathsf{(1)-(3)}$  do not necessarily imply that $(\mathsf{X}, h)$ is  Einstein or Ricci-flat, by   \cite[p.~11]{LorEin} it follows that that the only non-zero value of the
  Ricci tensor $\Ric^{h}$  is given by
\begin{equation}\label{riccianton}
\Ric^{h}_{uu}=-\frac{1}{2}\Delta H\,,
\end{equation}
where $\Delta
H=\sum_{i,j=1}^4\rho^{ij}(\partial_i\partial_jH-\sum_{k=1}^4\Gamma_{ij}^k\partial_kH)$
is the Laplace-Beltrami operator  of the metric $\rho$ applied to $H$.  Consequently, the Ricci tensor of $(\mathsf{X}, h)$ inherits the remarkable  property of {\it nullness}, in particular $\Ric^{h}$  must be  {\it totally isotropic}. Recall that 
\begin{definition}
A Lorentzian manifold $(\mathsf{X}, h)$ is called {\it totally Ricci-isotropic} if the Ricci endomorphism ${\rm ric}^{h}  : T\mathsf{X}\to T\mathsf{X}$ corresponding to $\Ric^{h}$ satisfies the relation
\[
h(\ric^{h}(X), \ric^{h}(Y))=0\,,
\]
for any $X, Y\in\fr{X}(\mathsf{X})$.
\end{definition}

 Lorentzian manifolds   admitting a parallel distribution of isotropic lines, which are totally Ricci-isotropic, share the same holonomy algebras with Ricci-flat Lorentzian manifolds \cite{Gal}.  Moreover, if  $(\mathsf{X}, h)$ is a spin Lorentzian manifold admitting a parallel spinor, then $h$ is totally Ricci isotropic (but not necessarily Ricci-flat, as in the Riemannian case). By some abuse of notation,  next we shall refer to Walker metrics satisfying the conditions $\mathsf{(1)-(3)}$  by the term   {\it Ricci-isotropic Walker metrics}, referring to the property that the image of the Ricci endomorphism  related to the Walker metric is totally null.  
 
 \begin{example}
 A special class of  Ricci-isotropic Walker manifolds are the so-called pp-waves, which satisfy the conditions  $\rho=\sum_{i=1}^{n}({\rm d}x^{i})^2$, $A=0$ and $\partial_{v}H=0$. 
Thus, in this case $\rho$ is a flat metric and the ambient space  $\mathsf{X}=\bb{R}\times\bb{R}^n\times\bb{R}\simeq\bb{R}^{n+2}$  is endowed with a Walker metric  of the form  $h=2{\rm d}v{\rm d}u+\sum_{i=1}^{n}({\rm d}x^{i})^{2}+H(x^1, \ldots, x^{n}, u)({\rm d}u)^{2}$, where now $H$ is $v$-independent.   In other words,  pp-waves are precisely the Walker manifolds whose holonomy algebra is commutative, $\fr{g}\subset\bb{R}^n\subset\fr{sim}(n)$.    Of course,  pp-waves are {totally Ricci-isotropic} (and hence with zero scalar curvature), and by (\ref{riccianton}) it follows that  the Ricci-flat condition  is equivalent to the condition  $\Delta H=0$. An important example of
geodesically complete pp-waves  are the Lorentzian symmetric spaces with solvable
transvection group, the so-called {\it Cahen-Wallach spaces} (due to  \cite{CW}).  In this case the function $H$ is given by $H(x^{1}, \ldots, x^{n-2}, u)=\sum_{i, j}\mc{A}_{ij}x^{i}x^{j}$ 
for some constant symmetric matrix $\mc{A}_{ij}$. Such metrics are indecomposable,  if and only if $\mc{A}$ is non-degenerate. 
 \end{example}

Below we will show that  the new  bosonic supergravity backgrounds  are established in terms of   some  six-dimensional  Ricci-isotropic Walker manifold and some five-dimensional Ricci-flat Riemannian manifold, i.e.
\begin{itemize}
\item   $\wi{M}^{1, 5}=\bb{R}\times N^{4}\times\bb{R}$, endowed with the Ricci-isotropic Walker  metric 
\begin{equation}\label{ppwave}
\tilde g=2{\rm d}v{\rm  d}u+\rho+H({\rm d}u)^2\,,
\end{equation}
 where $(N^4, \rho)$ is  a  four-dimensional  Ricci-flat Riemannian manifold and $H=H(x^1,\dots,x^4, u)$.
\item $(M^5, g)$ is a Ricci-flat Riemannian manifold (usually a product of a four-dimensional Ricci-flat K\"ahler  manifold $(\mc{P}^{4}, \mathsf{p}, J)$ with $\bb{R}$).
\end{itemize}
In particular, we analyse  how  such manifolds can be adapted to the  fixed special flux form induced by $\mathsf{F}$, in order to produce (via a methodological approach) explicit solutions of the eleven-dimensional supergravity equations of motion, which are based on our Ansatz. 
At this point, we should mention the following observation: 
By  \cite{Fig3} it is known that (decomposable in our terms)  eleven-dimensional {\it symmetric} supergravity backgrounds  are written as $\mathsf{X}^{1, 10}=\wi{M}^{1, d-1}\times M^{11-d}$,
where $\wi{M}^{1, d-1}$ is a $d$-dimensional indecomposable  Lorentzian symmetric space, and $M^{11-d}$ is some $(11-d)$-dimensional irreducible Riemannian symmetric space (or a product of irreducible symmetric spaces).  For this reason, and since we  use non-symmetric Ricci-isotropic Walker metrics on $\wi{M}^{1, 5}$ and   Ricci-flat Riemannian metrics on $M^5$,  it follows that the given supergravity backgrounds are {\it not}  symmetric.


\subsection{Results related with the flux form $\mathsf{F}=\tilde\al$}\label{subs1}
Assume that the flux form $\mathsf{F}$ is given by  (\ref{spcasealpha}), i.e.  $\mathsf{F}:=\tilde\al$, 
 for some non-trivial  4-form $\tilde\al$ on $\wi{M}^{1, 5}$. When $\tilde\al$ is null and for the unwarped metric ($f=1$), by Proposition \ref{spck} and Corollary \ref{spcasealphaa} we obtain the following 
\begin{theorem}\label{thmnew1}
Consider a five-dimensional Riemannian Ricci flat manifold $(M^{5}, g)$ and a six -dimensional Lorentzian manifold $(\wi{M}^{1, 5}, \tilde{g})$ endowed with a null 4-form $\tilde\al$ satisfying $\mathrm{d} \tilde \al = \mathrm{d} \tilde * \tilde \al =0$ and  
\[
\Ric^{\tilde g}(\tilde X , \tilde Y)=-\frac{1}{2}\langle i_{\tilde X}\tilde{\al}, i_{\tilde Y}\tilde{\al}\rangle_{\tilde g}\,, 
\]
 for any $\tilde X, \tilde Y \in \mathfrak{X}  (\wi M)$.  Then, the triple $(\mathsf{X}^{1, 10}=M^{5}\times \wi{M}^{1, 5}, \ h=g+\tilde g, \ \mathsf{F}=\tilde\al)$ is an eleven-dimensional  bosonic supergravity background. 
\end{theorem}

Let us describe some supergravity backgrounds related to Thereom \ref{thmnew1}.
\begin{corol}\label{sol1} 
Let $(\wi{M}^{1, 5}=\bb{R}\times N^{4}\times\bb{R}, \tilde{g}=2{\rm d}v{\rm  d}u+\rho+H({\rm d}u)^2)$ be a six-dimensional  Ricci-isotropic Walker manifold.
  Let $\theta\in\Omega^3(N)$ be a closed and co-closed 3-form on  $N^4$.   Then the  direct product of $\wi{M}^{1, 5}$ with a  five-dimensional Ricci-flat Riemannian manifold $(M^{5}, g)$,  induces solutions of the bosonic supergravity equations for the flux form   $\mathsf{F}=\tilde\al:={\rm d}u \wedge \theta$, if and only if  
 \begin{equation}\label{hessss}
  \Delta H=\|\theta\|^{2}_{\rho}\,.
  \end{equation}
\end{corol}
\begin{proof}
Set $\tilde\al:={\rm d}u \wedge \theta$, 
 where
$\theta$ is a 3-form on $N^4$.    Since the 1-form
${\rm d}u$ is null, $\tilde{\al}$ is a null 4-form on $\wi{M}^{1, 5}$. 
Moreover, $\tilde\al$ satisfies  the relation   
\[
\tilde{\ast}\tilde{\al} =  {\rm d}u\wedge\ast_{\rho}\theta\,,
\] 
where  $\ast_\rho$ denotes the Hodge operator on $(N^{4}, \rho)$.  Consequently, assuming that  $\theta$ is
  closed and co-closed, we obtain the first desired condition of Theorem \ref{thmnew1}, i.e. $\mathrm{d}\tilde\al=0= \mathrm{d} \tilde * \tilde \al$.  
 By Theorem \ref{thmnew1} we  also know that  the supergravity Einstein equation is given by     $\Ric^{\tilde g}(\tilde X , \tilde Y) =-\frac{1}{2}\langle i_{\tilde X}\tilde{\al}, i_{\tilde Y}\tilde{\al}\rangle_{\tilde g}$, and according to  (\ref{riccianton}) the  Ricci tensor  $\Ric^{\tilde g}(\tilde X , \tilde Y)$ is non-zero only in the direction of $\tilde X=\partial/\partial u$. In particular, we compute
      \[
  \Ric^{\tilde g}(\tilde X , \tilde X) =-\frac{1}{2}\langle i_{\tilde X}\tilde{\al}, i_{\tilde X}\tilde{\al}\rangle_{\tilde g}=-\frac{1}{2\cdot 3!}\tilde {g}(i_{\tilde X}\tilde{\al}, i_{\tilde X}\tilde{\al})=-\frac{1}{12}\rho(\theta, \theta)=-\frac{1}{2}\|\theta\|^{2}_{\rho}\,,
  \]
and hence, a comparison      with   (\ref{riccianton}) gives the presented condition, i.e. $\Delta H=\frac{1}{6}\rho(\theta,\theta)=\langle\theta, \theta\rangle_{\rho}=\|\theta\|^{2}_{\rho}$.   In the direction of  some other vector field $ \mathfrak{X}  (\wi M)\ni\tilde X\neq \partial/\partial u$,  the  supergravity Einstein equation reduces to the constraint $\|i_{\tilde X}\tilde\al\|^2_{\tilde g}=0$, which is satisfied since  $i_{\tilde X}\tilde\al=-{\rm d}u\wedge i_{\tilde X}\theta$, i.e. $i_{\tilde X}\tilde\al$ is  a null form.  
  \end{proof}
  \begin{example}
  For an explicit example of Corollary \ref{sol1}, set $N=\bb{R}^4$. We need to specify a closed and co-closed 3-form $\theta\in\Omega^{3}(\bb{R}^4)$ and   a smooth function $H$ satisfying (\ref{hessss}). Set  $ \theta={\rm d}x^{i}\wedge{\rm d}x^{j}\wedge{\rm d}x^{k}$,
    for some $1\leq i<j<k\leq 4$, which is  closed and co-closed.  
     Then we get solutions  of (\ref{hessss})  for 
$H=\frac{1}{8}\big((x^{1})^{2}+\cdots+(x^{4})^{2}\big)$. Note that  $H$ depends on all the variables $\{x^{1}, \ldots x^{4}\}$, thus the used metric $\tilde g$  is indecomposable.
  \end{example}
   
 \begin{remark}\label{similar}
 The list of maximally supersymmetric backgrounds of 11d-supergravity includes the flat Minkowski space, the products $\Ad S_4\times \Ss^7$, $\Ad S_7\times\Ss^4$ and a pp-wave solution, which was discovered by J.~Kowalski-Glikman \cite{Kow} (see also \cite{Hul}).  For this solution  the related bosonic fields are given by the eleven-dimensional pp-wave metric  $h=2{\rm d}v{\rm d}u+\rho+H(x^i, u)({\rm d}u)^2$,    
 where $H$ depends on $u$ and the local coordinates  $\{x^i : i=1, \ldots, 9\}$ of $\bb{R}^9$, and the flux form  which has the form
 $\mathsf{F}={\rm d}u\wedge\theta$, where $\theta$ is a closed and co-closed 3-form on $\bb{R}^9$, satisfying  $ \Delta H=-\|\theta\|^2$. 	Thus, in the context  of  bosonic supergravity backgrounds, our Corollary \ref{sol1}  provides an analogue of this construction, expressed  however in terns of more general six-dimensional Ricci-isotropic Walker manifolds and five-dimensional Ricci-flat manifolds.  
   \end{remark}

 
 \subsection{Results related with the flux form $\mathsf{F} =\tilde \be \wedge \nu$.}\label{subs2}
Let us assume that  the flux 4-form is given by  (\ref{spcasebeta}), i.e. 
\[
\mathsf{F} =\tilde \be \wedge \nu\,, \quad \tilde\be\in\Omega^{3}(\wi{M})\,, \quad \nu\in\Omega^{1}(M)\,.
\]
When $\tilde\be$ is null and   $h$ is the product metric ($f=1$), by Proposition \ref{spck} and Corollary \ref{spcasebetaa}  we result with the following
\begin{theorem}\label{newthem2}
Let $(M^{5}, g)$ be a  five-dimensional Riemannian Ricci flat manifold   and $(\wi{M}^{1, 5}, \tilde{g})$  be a six-dimensional Lorentzian manifold. Assume that $\nu$  is a non-trivial 1-form on $M^{5}$ of unit length, $\|\nu\|_{g}^{2}=-1$, and   $\tilde\be$  is a non-trivial null 3-form on $\wi{M}^{1, 5}$, satisfying the following equations
\begin{eqnarray*}
\mathrm{d}  \nu   &=& \mathrm{d}  * \nu   =0, \quad \mathrm{d} \tilde \be  = \mathrm{d} \tilde * \tilde \be  =0, \\
\Ric^{\tilde g}(\tilde X , \tilde Y) &=&\frac{1}{2}\langle i_{\tilde X}\tilde{\be}, i_{\tilde Y}\tilde{\be}\rangle_{\tilde g}, \quad \forall \ \tilde X, \tilde Y \in \mathfrak{X}  (\wi M).
\end{eqnarray*}
Then the product manifold $(\mathsf{X}^{1, 10}=M\times\wi{M}, h=g+\tilde g, \mathsf{F}=\tilde \be \wedge \nu)$ induces a bosonic supergravity background. 
\end{theorem}
To illustrate this statement  by examples we proceed as follows.
\begin{corol}\label{sol2}
Let $(\wi{M}^{1, 5}=\bb{R}\times N^{4}\times\bb{R}, \tilde{g}=2{\rm d}v{\rm  d}u+\rho+H({\rm d}u)^2)$ be a six-dimensional  Ricci-isotropic Walker manifold,
 and assume that $\omega\in\Omega^{2}(N^4)$ is a closed and co-closed 2-form on  $N^4$.  Let also $(M^{5}=\mc{P}^{4}\times\bb{R}, g)$ be a five-dimensional Riemannian manifold, where the factor  $(\mc{P}^4, \mathsf{p})$ is   Ricci-flat     and   $g=-\mathsf{p}-{\rm d}t^2$.
Then, the product $(\wi{M}^{1, 5}\times M^5, h=\tilde{g}+g, \mathsf{F}=\tilde \be \wedge\nu)$ is a bosonic supergravity background, with $\tilde\be:={\rm d}u \wedge \omega\in\Omega^{3}(\wi{M}^{1, 5})$ and $\nu:={\rm d}t\in\Omega^{1}(M^{5})$, if and only if  
\begin{equation}\label{hessss2}
\Delta H=-\|\omega\|^{2}_{\rho}\,.
\end{equation} 
\end{corol}
\begin{proof}
Consider the 3-form $\tilde\be:={\rm d}u \wedge \omega$, 
 where 
$\omega$ is a 2-form on $(N^{4}, \rho)$. Then $\tilde\be$ is a null 3-form on $\wi{M}^{1, 5}$  satisfying
\[
\tilde\ast({\rm d}u \wedge \omega) ={\rm d}u\wedge *_\rho\omega\,.
\] 
  Consequently, if  $\omega$ is closed and co-closed form on $N$, then the conditions given in  Theorem \ref{newthem2} for $\tilde\beta$   are satisfied. Now, by Theorem \ref{newthem2} and    (\ref{riccianton}) we deduce that the supergravity Einstein equation induces a non-trivial condition only in the direction of  $\tilde X=\partial/\partial u$, which gives   \[
  \Ric^{\tilde g}(\tilde X , \tilde X) =\frac{1}{2}\langle i_{\tilde X}\tilde{\be}, i_{\tilde Y}\tilde{\be}\rangle_{\tilde g}=\frac{1}{2\cdot 2!}\tilde{g}(i_{\tilde X}\tilde{\be}, i_{\tilde Y}\tilde{\be})=\frac{1}{4}\rho(\omega, \omega)=\frac{1}{2}\|\omega\|^{2}_{\rho}\,.
  \]
  Thus, a comparison with (\ref{riccianton})  induces   (\ref{hessss2}).  For some $\mathfrak{X}  (\wi M)\ni \tilde X\neq \partial/\partial u$ we obtain the condition $\|i_{\tilde X}\tilde\beta\|_{\tilde g}^2=0$ and this is trivially satisfied since  $i_{\tilde X}\tilde\beta=-{\rm d}u\wedge i_{\tilde X}\omega$, which is null.
\end{proof}
\begin{example}
To construct examples of Corollary \ref{sol2}, we may assyme that $N^{4}$ is a   Ricci-flat K\"ahler manifold and identify $\omega$  with  the K\"ahler
form (which is  closed and co-closed).   For instance, take $(N=\bb{R}^{4}, \rho, J)$ and let $\{e_1, \ldots, e_4\}$ be an orthonormal basis. The K\"ahler form is given by  $\omega={\rm d}x^{1}\wedge {\rm d}x^{2}+{\rm d}x^{3}\wedge dx^{4}$.  	Then we have  $\omega\wedge\omega=2\vol_{\rho}$, where $\vol_{\rho}$ denotes the volume element of $(N^{4}, \rho)$ and because $\omega=\ast_{\rho}\omega$, by the definition of the Hodge star operator we get 
\[
2\vol_{\rho}=\omega\wedge\omega=\omega\wedge\ast_{\rho}\omega=\langle\omega, \omega\rangle_{\rho}\vol_{\rho}=\|\omega\|^{2}_{\rho}\vol_{\rho}\,.
\]
Thus $\|\omega\|^{2}_{\rho}=2$ and  the relation (\ref{hessss2}) reduces to   $\Delta H=-2$, which is  satisfied for  $H=\frac{1}{4}((x^{1})^{2}+\cdots + (x^{4})^{2})$, for example.
\end{example}


\subsection{Results related with the flux form $\mathsf{F} = \tilde \varpi \wedge \eps $} \label{subs4} Assume now that $\mathsf{F}$ is given by (\ref{spcasevarpi}),
\[
\mathsf{F} = \tilde \varpi \wedge \eps\,,\quad \tilde \varpi\in\Omega^{1}(\wi{M})\,, \quad \eps\in\Omega^{3}(M)\,.
\]
   When $f=1$ and $\tilde\varpi$ is null, $\|\tilde\varpi\|^{2}_{\tilde g}=0$, a combination of Proposition \ref{spck} and Corollary \ref{spcasevarpii}  yields the following  
   \begin{theorem}\label{null13form}
Let $(M^{5}, g)$ be a  five-dimensional Riemannian Ricci flat manifold   and $(\wi{M}^{1, 5}, \tilde{g})$ be   a six-dimensional Lorentzian manifold.  Assume that $\eps$ is a 3-form on $M^{5}$ with constant length $\|\eps\|^{2}_{g}=-2$ and $\varpi$ a null 1-form on $\wi{M}$ satisfying the following conditions
\begin{eqnarray}
\mathrm{d} \tilde \varpi   &=& \mathrm{d} \tilde * \tilde \varpi   =0\,, \quad \mathrm{d} \eps     = \mathrm{d}   * \eps =0\,, \label{cclos}\\
\Ric^{\tilde g}(\tilde X , \tilde Y) &=&   \tilde \varpi(\tilde X) \cdot\tilde \varpi(\tilde Y)=(\tilde\varpi\otimes\tilde\varpi)(\tilde X,\tilde Y)\,, \quad \forall \   \ \tilde X\,, \tilde Y \in \mathfrak{X}  (\wi M)\,.  \label{ricpal}
\end{eqnarray}
Then the product manifold $(\mathsf{X}^{1, 10}=M\times\wi{M}, h=g+\tilde g)$ endowed with the flux 4-form $\mathsf{F}=\tilde \varpi \wedge \eps$ solves the bosonic supergravity equations. 
\end{theorem}
Observe that in this case the  supergravity Einstein equation asserts to the Ricci tensor of the Lorentzian manifold $(\mathsf{X}^{1, 10}=M^5\times \wi{M}^{1, 5}, h=g+\tilde g)$  the property of totally nullness. 
\begin{corol}
The bosonic supergravity background $(\mathsf{X}^{1, 10}=M^5\times \wi{M}^{1, 5}, h=g+\tilde g, \mathsf{F} = \tilde \varpi \wedge \eps)$ given in Theorem \ref{null13form}, is totally Ricci isotropic.
\end{corol}
\begin{proof}
By assumption, $(M^5, g)$ is Ricci flat and hence it is sufficient to show that the  Lorentzian manifold $(\wi{M}^{1, 5}, \tilde{g})$   is   totally Ricci-isotropic (see for example \cite{LorEin}).   
The Ricci endomorphism  ${\rm ric}^{\tilde g} : T\wi{M}^{1, 5}\to T\wi{M}^{1, 5}$ corresponding to  $\tilde g$ is defined by   $\tilde{g}\big({\rm ric}^{\tilde g}(\tilde X), \tilde Y\big)=\Ric^{\tilde g}(\tilde X, \tilde Y)$,  for any $\tilde X\,, \tilde Y \in \mathfrak{X}  (\wi M)$, where as above $\Ric^{\tilde g}$ denotes the Ricci tensor.
 Then, by the relation (\ref{ricpal}) we deduce that 
 \[
 \tilde{g}\big({\rm ric}^{\tilde g}(\tilde X), \tilde Y\big)=\tilde \varpi(\tilde X)\tilde \varpi(\tilde Y)=\tilde \varpi(\tilde X)\tilde g(\tilde \varpi^{\sharp}, \tilde Y) =\tilde g(\tilde \varpi(\tilde X)\tilde \varpi^{\sharp}, \tilde Y)\,,
 \] 
and thus,   ${\rm ric}^{\tilde g}(\tilde X)=\tilde \varpi(\tilde X)\tilde \varpi^{\sharp}$.
Since  the  1-form $\tilde \varpi$ is by assumption   null,    so is $\tilde \varpi^{\sharp}$, i.e. $\tilde g(\tilde \varpi^{\sharp}, \tilde \varpi^{\sharp})=0$ and thus  $
 \tilde{g}\big({\rm ric}^{\tilde{g}}(\tilde X), {\rm ric}^{\tilde{g}}(\tilde Y)\big)=\tilde\varpi(\tilde X)\tilde\varpi(\tilde Y)\|\tilde\varpi^{\sharp}\|^{2}_{\tilde g}=0$, which proves the claim.
\end{proof} 
  Let us now describe supergravity solutions related to Theorem \ref{null13form}.
\begin{corol}\label{sol3}
Let $(\wi{M}^{1, 5}=\bb{R}\times N^{4}\times\bb{R}, \tilde{g}=2{\rm d}v{\rm  d}u+\rho+H({\rm d}u)^2)$ be a six-dimensional  Ricci-isotropic Walker manifold.
  Let also $(M^{5}=\mc{P}^{4}\times\bb{R}, g)$, (or $M^{5}=\mc{P}^{4}\times\Ss^1$)   be a   five-dimensional Riemannian manifold, where $(\mc{P}^{4}, \mathsf{p}, \omega, J)$ is a Ricci-flat K\"ahler manifold and $g=-\mathsf{p}-{\rm d}t^2$. Then,   $(\wi{M}^{1, 5}\times M^5, \tilde g+g)$ endowed with the flux form $\mathsf{F} = \tilde \varpi \wedge \eps$, where $ \tilde \varpi ={\rm d}u\in\Omega^{1}(\wi{M}^{1, 5})$ and $\eps=\omega\wedge {\rm d}t \in\Omega^{3}(M^{5})$, respectively, is a bosonic supergravity background, if and only if $\Delta H=-2$.
\end{corol}
\begin{proof}
Consider   the Lorentzian manifold $\wi{M}^{1, 5}=\bb{R}\times N\times \bb{R}$ endowed with the metric $\tilde g$ given in (\ref{ppwave}), and  set  $\tilde\varpi:={\rm d}u\,.$
Then, $\tilde\varpi$ is null and parallel, and hence closed and co-closed. On the other hand,  by Theorem \ref{null13form}  the supergravity Einstein equation has the form
\[
\Ric^{\tilde g}(\tilde X , \tilde X) =\tilde \varpi(\tilde X) \tilde \varpi(\tilde X)
\]
and according to (\ref{riccianton})  this survives only in the direction of  $\tilde X= \partial/\partial u$, for which  gives 
$\Ric^{\tilde g}(X, X)=1$ (for any other vector field $\tilde X$ both sides are zero). Hence, by  (\ref{riccianton}) we finally get the constrain $\Delta H=-2$.   For the Riemannian part of the direct product $\mathsf{X}^{1, 10}=M^{5}\times\wi{M}^{1, 5}$, according to Theorem \ref{null13form}, we need a Ricci-flat manifold $(M^{5}, g)$ with a closed and co-closed 3-form $\eps$ of constant length $-2$. Consider the   product $M^{5}=\mc{P}^{4}\times\bb{R}$ (or $\mc{P}^{4}\times\Ss^{1}$), where $(\mc{P}^{4}, \mathsf{p}, \omega, J)$ is some Ricci-flat four-dimensional K\"ahler manifold and set $ \eps:=\omega\wedge {\rm d}t$. 
Then, $\ast\eps=\ast_{\mathsf{p}}\omega$ and    since $\omega$ is a K\"ahler form, $\eps$ is closed and co-closed.  Moreover, $\|\eps\|^{2}_{g}=\|\omega\|_{\mathsf{p}}^{2}\|{\rm d}t\|^2_{-{\rm dt}^2}=-2$, as required.
\end{proof}

\begin{example}
A explicit example satisfying  Corrollary \ref{sol3} is given   by choosing $N^4=\bb{R}^4$ and $H=\frac{1}{4}((x^{1})^{2}+\cdots + (x^{4})^{2})$.
\end{example}

\subsection{Results related with the flux form $\mathsf{F} = \tilde \al + \tilde \be \wedge \nu$} \label{subs6}
 In the remainder of this section, we  assume that $\mathsf{F}$ is given by (\ref{spcasesumleft}), 
\[
\mathsf{F}:= \tilde \al + \tilde \be \wedge \nu\,, \quad \tilde\al\in\Omega^{4}(\tilde M)\,, \quad  \tilde\be\in\Omega^{3}(\tilde M)\,,  \quad \nu\in\Omega^{1}(M)\,.
\]
To simplify the situation in the constrains given in  Proposition \ref{spck}  we may set  $\con =1$.   Let us also assume that  $f=1$. Then, by  Proposition \ref{spck}  and Corollary \ref{spcasesumleftt}  we deduce the  following
\begin{theorem}\label{spckckck}
Let $(M^{5}, g)$ be a  five-dimensional Ricci-flat Riemannian manifold endowed with  a  closed 1-form $\nu$ such that $\mathrm{d} * \nu =\mathrm{vol} _M$. Then, the product of $M^{5}$ with a six-dimensional Lorentzian manifold $(\wi{M}^{1, 5}, \wi{g})$  satisfying the following system of equations 
\begin{eqnarray}
\Ric^{\tilde g} (\tilde X , \tilde Y) &=& -\frac{1}{2}\langle i_{\tilde X}\tilde{\al}, i_{\tilde Y}\tilde{\al}\rangle_{\tilde g}-\frac{\|\nu\| ^2 _{g}}{2}\langle i_{\tilde X}\tilde{\be}, i_{\tilde Y}\tilde{\be}\rangle_{\tilde g}\,, \label{riccitilda} \\
\mathrm{d}\tilde  \al &=& \mathrm{d} \tilde \be=\mathrm{d} \tilde * \tilde \be  =0\,,\quad \tilde * \tilde \be = - \mathrm{d} \tilde * \tilde \al\,,  \ \label{maxwelt} \\
 \langle i _{ \tilde X }\tilde \al, \tilde \be \rangle_{\tilde g}&=&0\,, \ \label{extra} 
\end{eqnarray}
for some null forms $\tilde\al\in\Omega^{4}(\wi M^{1, 5})$ and $\tilde\be\in\Omega^{3}(\wi M^{1, 5})$ and for any $\tilde X, \tilde Y\in\fr{X}(\wi{M}^{1, 5})$, solves the bosonic supergravity equations for the flux form $\mathsf{F}:= \tilde \al + \tilde \be \wedge \nu$. In particular, $(M^5, g)$ must be non-compact.
\end{theorem}
\begin{proof}
We have $\|\tilde \al \| _{ \tilde g }^2 =\|\tilde \be  \| _{ \tilde g }^2  =0$ by assumption, and all the constraints  are direct  by  Proposition \ref{spck}  and Corollary \ref{spcasesumleftt}, except the one for the non-compactness. 
Assume that $M^5$ is closed. Then, since $\ast\nu$ is a 4-form and  $M^5$ has no boundary,  the claim follows by Stokes theorem, 
\[
\int_{M}{\rm d}\ast\nu=\int_{\partial M}\ast\nu=0\,,
\]
a contradiction due to the constrain which establishes the Maxwell equation, i.e. $\mathrm{d} * \nu = \mathrm{vol} _M$\,.
\end{proof}

\begin{remark}
\textnormal{Based again on Stokes theorem and regarding  this time the flux-form $\mathsf{F}$ defined by $\mathsf{F}=\tilde \varpi \wedge \eps  + \theta$, 
one similarly deduces  that the Lorentzian manifold $(\wi{M}^{1, 5}, \tilde g)$ must be non-compact (due to the condition $\mathrm{d} \tilde * \tilde \varpi  = \con   \mathrm{vol} _{\wi M}$, see Case (7) of  Proposition \ref{spck}). Note that Cases (3) and (7) of  Proposition \ref{spck} will not be further analysed in the present work, while  results for Case (8) (respectively  Case(9)), are extracted by the  study given in  subsection \ref{subs1} and in the last  section (respectively, in subsection \ref{subs4}).}
\end{remark}

Let us present particular  examples  satisfying Theorem \ref{spckckck}. The idea here occurs as a  combination of the constructions discussed  in Corollaries \ref{sol1} and \ref{sol2}, respectively. At a fist step we provide the following more general result.
\begin{prop}\label{sol4}
Let $(\wi{M}^{1, 5}=\bb{R}\times N^{4}\times\bb{R}, \tilde{g}=2{\rm d}v{\rm  d}u+\rho+H({\rm d}u)^2)$ be a six-dimensional  Ricci-isotropic Walker manifold.
   Set
\[
\tilde\al:={\rm d}u\wedge\Omega\,, \quad \tilde\be:={\rm d}u\wedge\omega,
\]
for some    closed 3-form $\Omega\in\Omega^{3}(N^{4})$ and a closed  2-form $\omega\in\Omega^{2}(N^{4})$, such that
\begin{equation}\label{omeome}
\omega=\ast_{\rho}{\rm d}\ast_{\rho}\Omega\,.
\end{equation}
Let also $(M^{5}, g)$  be a non-compact  five-dimensional Ricci-flat Riemannian manifold endowed with a closed 1-form  $\nu\in\Omega^{1}(M^{5})$,   such that $\mathrm{d} * \nu =\mathrm{vol} _M$.  Then, the eleven-dimensional  Lorentzian manifold   $(\mathsf{X}^{1, 10}=M^{5}\times\wi{M}^{1, 5}, h=g+\tilde g)$
is a bosonic supergravity background   with respect to the flux form $\mathsf{F}=\tilde\al+\tilde\be\wedge\nu$, 
if and only if  
\begin{equation}\label{finhese}
\Delta H=\|\Omega\|^{2}_{\rho}+\|\nu\|^{2}_{g}\cdot\|\omega\|^{2}_{\rho}\,.
\end{equation}
 \end{prop}
\begin{proof}
The  forms  $\tilde\al={\rm d}u\wedge\Omega$ and  $\tilde\be={\rm d}u\wedge\omega$ are 
  null forms on $(\wi{M}^{1, 5}, \tilde g)$.  Moreover, $\tilde\al$ (respectively $\tilde\be$) is closed, if and only if $\Omega$ (respectively $\omega$) is closed. By the proof in Corollary \ref{sol2} recall   that $\tilde\ast\tilde\be={\rm d}u\wedge\ast_{\rho}\omega\,,$
 hence  the  condition ${\rm d}\tilde\ast\tilde\be=0$ of   (\ref{maxwelt})   is satisfied,  if and only if 
\begin{equation}\label{coclosebe}
{\rm d}\ast_{\rho}\omega=0.
\end{equation}
The Hodge star $\tilde\ast\tilde\al$  is the 2-form on $\wi{M}^{1, 5}$  given by  $\tilde\ast\tilde\al={\rm d}u\wedge\ast_{\rho}\Omega\,,$
with  $\ast_{\rho}\Omega\in\Omega^{1}(N^{4})$. Hence,  the second condition    in  (\ref{maxwelt}), i.e. the relation $\tilde * \tilde \be = - \mathrm{d} \tilde * \tilde \al$, takes the form      
\[
{\rm d}u\wedge\ast_{\rho}\omega=-{\rm d}({\rm d}u\wedge\ast_{\rho}\Omega)={\rm d}u\wedge {\rm d}\ast_{\rho}\Omega.
\]
In this way we prove (\ref{omeome}), i.e.
 \[
\ast_{\rho}\omega= {\rm d}\ast_{\rho}\Omega \, \quad\Longleftrightarrow\quad \omega=\ast_{\rho}{\rm d}\ast_{\rho}\Omega\,,
\]
where the equivalence occurs  since $\ast_{\rho}^{2}$ acts as the identity on 2-forms on $N^{4}$.  Thus, observe  that if (\ref{omeome}) is satisfied, then also (\ref{coclosebe}) is valid, hence $\tilde\beta$   is co-closed, as required. Passing to the supergravity Einstein equation, as in the previous cases, by Theorem \ref{spckckck} and (\ref{riccianton}) we deduce that this is non-zero only in the direction of   $\tilde X=\partial/\partial u$. Indeed,   by  (\ref{riccitilda})   we see that
\[
\Ric^{\tilde g}(\tilde X, \tilde X)=-\frac{1}{12}\rho(\Omega, \Omega)-\frac{\|\nu\|^{2}_{g}}{4}\rho(\omega, \omega)\,,
\]
and a   comparison   with the relation (\ref{hessss2}) gives
 \[
 \Delta H=\frac{1}{6}\rho(\Omega, \Omega)+\frac{\|\nu\|^{2}_{g}}{2}\rho(\omega, \omega)=\langle\Omega, \Omega\rangle_{\rho}+\|\nu\|^{2}_{g}\langle\omega, \omega\rangle_{\rho}\,,
 \]
 which is equivalent to    (\ref{finhese}).  For some vector field $\tilde X\in\fr{X}(\widetilde M)$ different than $\partial/\partial u$ we see that $i_{\tilde X}\tilde\al=-{\rm d}u\wedge i_{\tilde X}\Omega$ and $i_{\tilde X}\tilde\be=-{\rm d}u\wedge i_{\tilde X}\omega$, which are both null and hence the induced condition by the supergravity Einstein equation is trivially satisfied.  To  complete the proof we need to show that
 the relation (\ref{extra}) is also true,  which easily follows.  
 \end{proof}
At a first sight,   the conditions (\ref{omeome}) and (\ref{finhese}) may seem complicated.  However,  since $N$ can be assumed to be  flat  it is not so hard to construct explicit examples.  
\begin{example}
\textnormal{Let us illustrate this construction in terms of specific pp-waves. So, set $N=\bb{R}^{4}$ and let  $\vol_{\rho}={\rm d}x^{1}\wedge {\rm d}x^{2}\wedge {\rm d}x^{3}\wedge {\rm d}x^{4}$ be the volume element.   For a 3-form we may use
\begin{equation}\label{OMEg}
 \Omega^{+}:= x^{2}{\rm d}x^{2}\wedge{\rm d}x^{3}\wedge{\rm d}x^{4}\,.
\end{equation}
Then ${\rm d}\Omega^{+}=0$ and moreover 
\[
 \ast_{\rho}\Omega^{+}=-x^{2}{\rm d}x^{1}\,\quad\Longrightarrow \quad {\rm d}\ast_{\rho}\Omega^{+}={\rm d}x^{1}\wedge{\rm d}x^{2}\,\quad\Longrightarrow \quad\ast_{\rho}{\rm d}\ast_{\rho}\Omega^{+}={\rm d}x^{3}\wedge{\rm d}x^{4}\,,
\]
and hence in a line with (\ref{omeome}) we can set
\begin{equation}\label{omeome1}
\omega^{+}:=\ast_{\rho}{\rm d}\ast_{\rho}\Omega^{+}={\rm d}x^{3}\wedge{\rm d}x^{4}\,.
\end{equation}
Then, ${\rm d}\omega^{+}=0$ and moreover ${\rm d}\ast_{\rho}\omega^{+}=0$, hence the condition (\ref{coclosebe}) is satisfied.  In particular, we have constructed a pair  $(\Omega^{+}, \omega^{+})\in\Omega^{3}(N)\times\Omega^{2}(N)$
such that the   null forms $\tilde\al^+:={\rm d}u\wedge\Omega^+$ and $\tilde\be^+:={\rm d}u\wedge\omega^+$ are solutions of  the  system of equations given in (\ref{maxwelt}). Moreover, $\tilde\al^+$ and $\tilde\be^+$ satisfy the relation (\ref{extra}) as an identity, for any $\tilde X\in\fr{X}(\widetilde M)$.  For the Riemannian part we need a non-compact Ricci-flat manifold $(M^{5}, g)$ admitting a closed 1-form $\nu\in\Omega_{\rm cl}^{1}(M)$ such that ${\rm d}\ast\nu=\vol_{M}$. Consider   $M^{5}=\bb{R}^{5}$ and the 1-form $\nu=y^{1}{\rm d}y^{1}$. Then $\nu$ is closed with $\ast\nu=y^{1}{\rm d}y^{2}\wedge{\rm d}y^{3}\wedge{\rm d}y^{4}\wedge {\rm d}y^{5}$ and  ${\rm d}\ast\nu=\vol_{M}$.  One computes   $\|\Omega\|^{2}_{\rho}=-(x^{2})^{2}$, $\|\omega\|^{2}_{\rho}=1$ and $\|\nu\|^{2}_{g}=-(y^{1})^{2}$.  Thus, a function $H$ solving (\ref{finhese}) is given by $H=\frac{1}{12}\big((x^{1})^{4}+(x^{2})^{4})$, and since $H$   depends only on $x^{1}, x^{2}$,  this example produces decomposable solutions.   Note that analogous  results make sense also for the pair $(\Omega^{-}, \omega^{-})$, where the 3-form $\Omega^{-}$ is defined by $\Omega^{-}:= -x^{2}{\rm d}x^{2}\wedge{\rm d}x^{3}\wedge{\rm d}x^{4}$. More general, the same approach applies for any 3-form  of the type $\Omega=\pm x^{i}{\rm d}x^{i}\wedge{\rm d}x^{j}\wedge{\rm d}x^{k}\,,$
with $1\leq i<j<k\leq 4$.}
\end{example}

\section{Constructions from almost contact geometry}
In the final section, we focus on Case (5) of Proposition \ref{spck}, i.e. assume that   the flux 4-form depends only in the Riemmannian part and hence it is given by   (\ref{spcasetheta}), i.e.  $\mathsf{F}:=\theta$, 
 for some  non-trivial 4-form $\theta\in\Omega^{4}(M)$.  For this choice,  next  we describe constructions  related  with well-known  structures in Riemannian geometry,  e.g.    {\it almost contact structures}. 
 
First notice that for $f=1$, a combination of Proposition \ref{spck} and Corollary \ref{spcasethetaa}  yields that
\begin{theorem}\label{ineee}
Consider a five-dimensional Riemannian manifold $(M, g)$ endowed with a non-trivial 4-form $\theta$   of constant length, satisfying the following conditions
\begin{eqnarray*}
\dd\theta&=&\dd*\theta=0, \\
\Ric^{g}(X , Y ) &=& \frac{\|\theta\|^{2}_{g}}{6}g(X, Y)-\frac{1}{2}\langle i_{X}\theta, i_{Y}\theta\rangle_{g}\,, \quad \forall \    X\,, Y \in \mathfrak{X}  (M)\,.
\end{eqnarray*}
 Then,   the product  manifold $(\mathsf{X}^{1, 10}=M^{5}\times \wi{M}^{1, 5}, h=g+\tilde g)$ endowed with the flux 4-form $\mathsf{F}=\theta$,  where  $(\wi{M}^{1, 5}, \tilde g)$ is a six-dimensional Lorentzian Einstein manifold with Einstein  constant $\|\theta\|^{2}/6$,  solves the bosonic supergravity equations.
\end{theorem}
By using the Hodge star operator $\ast : \Omega^{p}(M)\to\Omega^{5-p}(M)$, one can rewrite the previous result in terms of the 1-form $\eta:=\ast\theta$, which turns out to be more useful for geometric applications.  Recall that   $X^{\flat}\wedge\ast \al= (-1)^{p-1}\ast(X\lrcorner \al)\,,$
   for any  $p$-form $\al$ and vector field $X$, where  $X^{\flat}$ is the dual 1-form, i.e. $X^{\flat}(Y)=g(X, Y)$.  For convenience, let us denote by $Z$ the vector field corresponding to $\eta$. 
    Then we obtain   $ \ast(Z\lrcorner \vol_{M})=\eta\wedge\ast \vol_{M}= \eta\,.$
Since $\ast^{2}\al=(-1)^{p(5-p)}\al$ for any $\al\in\Omega^{p}(M^{5})$, it finally follows that      $Z\lrcorner \vol_{M}=\ast\eta=\theta$.   Thus  based on the formula $\mc{L}_{Z}\vol_{M}={\rm div}(Z)\vol_{M}$ we deduce that
  \begin{lemma}\label{cmax}
  The condition  $  \dd\theta=\dd*\theta=0$ 
  is equivalent to say that  $\eta$ is a closed and co-closed 1-form, in particular the co-closedness  is equivalent to say that the  vector field $Z$ is divergence free.
  \end{lemma}
  Consequently,  one may construct  several examples of Riemannian manifolds $(M^{5}, g)$ admitting a closed and co-closed 4-form $\mathsf{F}:=\ast\eta=:\theta$  (and hence satisfy the first two conditions of the system of the eleven-dimensional bosonic supergravity equations). For instance, it is very easy to see that such a class is given in terms of  harmonic  maps $f : M^{5}\to\bb{R}$ and the 1-form $\theta={\rm d}f$.

Another class of Riemannian manifolds satisfying the condition ${\rm d}\eta=\dd\ast\eta=0$ for some 1-form $\eta$ arises in almost contact geometry. For such a description, we should  first recall a few  basic facts  (see  \cite{Blair} for an introduction to almost contact metric structures).
Let $(M^{2n+1}, g, \phi, \eta, \xi)$ be an almost contact metric manifold and let us denote by $\Phi(X, Y):=g(X, \phi Y)$ the fundamental 2-form, with  $\xi\lrcorner\Phi=0$. The Nijenhuis tensor associated to the endomorphism $\phi : TM\to TM$  is given by
 \[
 N_{\phi}(X, Y):=[\phi(X), \phi(Y)]+\phi^{2}([X, Y])-\phi([\phi(X), Y])-\phi([X, \phi(Y)])+\dd\eta(X, Y)\xi\,.
 \]
 $(M^{2n+1}, g, \xi, \eta, \phi)$ is   called {\it normal} if $N_{\phi}=0$ identically, and  {\it almost cosymplectic}, or {\it almost co-K\"ahler}  if  $\dd\Phi=0$ and $\dd\eta=0$. An almost cosymplectic manifold which is  normal is called {\it cosymplectic} or {\it co-K\"ahler}.  This is equivalent  to say  that   both $\eta$ and $\phi$ are $\nabla^{g}$-parallel, $\nabla^{g}\eta=\nabla^{g}\phi=0$. In this case the Reeb vector field $\xi$ is  Killing (notice that any Killing vector field is divergence free).  Finally,  $(M^{2n+1}, g, \xi, \eta, \phi)$ is said to be {\it nearly cosymplectic}, if $(\nabla^{g}_{X}\phi)(X)=0$ for any $X\in\Gamma(TM)$.
  \begin{prop}\textnormal{(\cite{GY, Blair})}
 1) On an almost cosymplectic manifold $(M^{2n+1}, g, \xi, \eta, \phi)$ the 1-form $\eta$ is closed and co-closed and the same property satisfies the fundamental 2-form $\Phi$.\\ 
 2) On a nearly cosymplectic manifold $(M^{2n+1}, g, \xi, \eta, \phi)$, the Reeb vector field is Killing. If in addition $N_{\phi}=0$, then $\eta$ is closed, ${\rm d}\eta=0$. In particular, a normal nearly cosymplectic manifold is cosymplectic.
\end{prop}
The consequences of  this proposition in our supergravity framework read as follows.
\begin{prop}
Let $(\mathsf{X}^{1, 10}=M^{5}\times\wi{M}^{1, 5}, h=g\times\wi{g})$ be a Lorentzian manifold given by the product of a six-dimensional  Lorentzian manifold $(\wi{M}^{5, 1}, \wi{g})$ and a five-dimensional   Riemannian manifold $(M^{5}, g, \xi, \eta, \phi)$, which is (almost) cosymplectic or  normal nearly cosymplectic. Then, the 4-form  $\mathsf{F}:=\ast\eta$ is closed and satisfies the Maxwell equation. 
\end{prop}

Now,  in terms of the 1-form $\eta=\ast\theta$ the supergravity Einstein equation and  in particular the conditions given in Theorem \ref{ineee}, can be rephrased as follows:  \\
\noindent $(i)$ The Lorentzian manifold $(\wi{M}^{5, 1}, \wi{g})$ is  Einstein, $\Ric^{\wi{g}}=\frac{1}{6}\|\eta\|^{2}_{g}\cdot\wi{g}.$\\
\noindent $(ii)$  The Ricci tensor $\Ric^{g}$ of $(M^{5}, g, \eta)$ satisfies the equation
\[ 
\Ric^{g}(X, Y)=-\frac{1}{3}g(X, Y)\|\eta\|_{g}^{2}+\frac{1}{2}\eta(X)\cdot\eta(Y)\,,
 \]
 for any $X, Y\in\fr{X}(M)$. To see this one may use  the relation $\|\theta\|^{2}_{g}=\|\eta\|^{2}_{g}$,  Theorem \ref{ineee} and \cite[Lem.~2.6]{ACT}. When $\eta$  is of unit length, then these conditions become even simpler.

 \begin{example}\label{nonexist}
 \textnormal{Consider a 4-dimensional almost Hermitian manifold $(B^{4}, \hat{g}, J)$.
  Then,  $M^{5}=B^{4}\times\bb{R}$ endowed with the product metric $g$,   admits an almost contact metric structure  which is given by $
 \xi=(0, \frac{\partial}{\partial t})$, $\eta=\dd t$ and $\phi(X, f \frac{\partial}{\partial t})=JX$,  
 for some smooth function $f : B^{4}\times\bb{R}\to\bb{R}$ and vector field $X\in\Gamma(TB)$, where we write $(X, f \frac{\partial}{\partial t})$ for a vector field on $M^{5}$, see \cite[p.~35]{Blair}. When the K\"ahler form $\hat{\omega}(X, Y)=\hat{g}(JX, Y)$ on $B$ is closed,  then the manifold $(M^{5}, g, \phi, \xi, \eta)$ is clearly an almost cosymplectic manifold. In a similar manner one can construct an almost cosymplectic structure  on the product $B^{4}\times\Ss^{1}$. If $J$ is integrable, then this structure is cosymplectic. Consider now an orthonormal basis $\{e_1,  \dots, e_5\}$ of  $T_{p}M$ at a point $p\in M^{5}$ such that $e_5:=\xi$.  Since $\xi$ is a vector field of unit length, so is $\eta$,  i.e. $\|\eta\|^{2}_{g}=-1$. Thus, by applying $(ii)$ we obtain $\Ric^{g}(e_{5}, e_{5})=\frac{1}{6}$. However,    the Ricci tensor of $(M^{5}=B^{4}\times\bb{R}, g, \xi, \eta, \phi)$ in the direction of the Reeb vector field $\xi$ must vanish (see for example \cite{Blair, Cappel}), i.e. $\Ric^{g}(e_{5}, e_{5})=0$, which shows that such a manifold cannot provide us with the desired supergravity solutions.}
 \end{example}
  
This construction generalizes in any odd dimension and it turns out that the simplest examples of almost cosymplectic manifolds $M^{2n+1}$ are products of almost K\"ahler manifolds with the real line or $\Ss^{1}$.  If $(M^{2n+1},  g, \xi, \eta, \phi)$  is   cosymplectic, then  locally it is always a product of a K\"ahler manifold with the real line or the circle.  But not any almost cosymplectic manifold  is of this type, even locally  (for details we refer to the survey  \cite{Cappel}).   For instance, due to the work of Olszak \cite{Olsz} are known examples of proper almost cosymplectic manifolds (Lie groups) whose Reeb vector field $\xi$ is not Killing and hence they are {\it not}  the product of an almost K\"ahler manifold with $\bb{R}$ (even locally).  However, running the supergravity Einstein equation   on such manifolds we again obtain incompatibility. In fact, this shows that in this case  the main obstruction to bosonic supergravity backgrounds is the supergravity Einstein equation.  Thus, it remains an open problem the construction of examples related to   Theorem \ref{ineee}.  

\medskip
 \noindent {\bf Acknowledgements:}  
 I.C. acknowledges full support via Czech Science Foundation (GA\v{C}R no.~19-14466Y).  The authors thank G. Franchetti and A. Santi for  helpful discussions related to the topic.


\end{document}